\documentclass[10pt]{article}

\usepackage{amsmath,amssymb,amsthm,mathrsfs,dsfont,graphicx}

\usepackage[margin=3cm]{geometry} 

\usepackage{titlesec,hyperref}
\usepackage{esint}
\usepackage{natbib}
\usepackage{color}

\usepackage{fancyhdr}
\titleformat{\subsection}{\it}{\thesubsection.\enspace}{1pt}{}

\newtheorem{theo}{Theorem}[section]

\newtheorem{rema}[theo]{Remark}

\newtheorem{lemm}[theo]{Lemma}
\newtheorem{defi}[theo]{Definition}

\newtheorem{prop}[theo]{Proposition}

\numberwithin{equation}{section}

\allowdisplaybreaks 

\begin{document}

\title{Non-uniqueness of smooth solutions of the 5D
magnetohydrodynamic equations from critical data
\hspace{-4mm}
}

\author{Zipeng $\mbox{Chen}^1$ \footnote{Email: chenzp26@mail2.sysu.edu.cn},\quad Song $\mbox{Liu}^1$ \footnote{Email: lius37@mail2.sysu.edu.cn}\quad
	 and\quad
	Zhaoyang $\mbox{Yin}^{1}$\footnote{E-mail: mcsyzy@mail.sysu.edu.cn}\\
    $^1\mbox{School}$ of Science,\\ Shenzhen Campus of Sun Yat-sen University, Shenzhen 518107, China}
 
\date{}
\maketitle
\hrule

\begin{abstract}
Recently, Coiculescu and Palasek \cite{Coiculescu2025} shows the non-uniqueness of solutions for the 3D incompressible Navier-Stokes equations with initial data in $BMO^{-1}$. Inspired by their breakthrough work, we develop their schemes for the incompressible magnetohydrodynamic equations and obtain a similar result in 5 dimensional case. More precisely, we construct two distinct global solutions with a initial data, which has nonvanishing velocity and magnetic fields in $BMO^{-1}(\mathbb{T}^5)$. 
\end{abstract}

\vskip 0.2cm


\vspace*{10pt}

\tableofcontents

\section{Introduction }

  In this paper, we consider the following $5$-dimensional  magnetohydrodynamic equation:
\begin{equation}\label{e:MHDe}
\begin{cases}
\partial_tu-\Delta u+ \text{div}(u\otimes u-b\otimes b)+\nabla p=0, \quad\quad \\
\partial_tb-\Delta b+ \text{div}(u\otimes b-b\otimes u)=0, \\
\text{div}\,u=\text{div}\,b=0,
\quad\forall(t,x)\in [0,\infty)\times\mathbb{T}^5,
\end{cases}
\end{equation}
where $\mathbb{T}^5$ is the $5$-dimensional toru. Here, $v,p,b$ represent the flow velocity, the scalar pressure, and the magnetic field, respectively.The magnetohydrodynamics equation describes the dynamic coupling between electrically conducting fluids and magnetic fields, unifying the principles of hydrodynamics and electromagnetism. Originating from astrophysical and geophysical studies, MHD equations model the motion of ionized fluids (plasmas, liquid metals, saltwater) where fluid flow induces magnetic field changes, and magnetic forces in turn drive fluid motion \cite{MHDphy1,MHDphy2}.

When $b = 0$, system (\ref{e:MHDe}) is reduced to the famous incompressible Navier–Stokes equations:
\begin{equation}\label{e:NS}
\begin{cases}
\partial_tv+ \text{div}(v\otimes v)+\nabla p-\Delta v=0, \\
\text{div}\,v=0,
\end{cases}
\end{equation}
In the remarkable paper \cite{leray}, for any dimension $d\geq2$, Leray first showed that there exists a weak solution in $L^\infty(\mathbb{R}^+;L^2(\mathbb{R}^d))\cap L^2(\mathbb{R}^+;\dot{H}^1(\mathbb{R}^d))$ for any $L^2$ solenoidal initial data, which satisfies the energy inequality
\begin{gather*}
    \|v(t)\|^2_{L^2}+2\int^t_0\|\nabla v(s)\|^2_{L^2}ds\leq\|v(0)\|^2_{L^2},\quad\forall t\geq0.
\end{gather*}
For smooth bounded domains with Dirichlet boundary conditions, Hopf \cite{hopf} derived an analogous result. This class of weak solutions is now known as Leray–Hopf weak solutions. For MHD equations, this class of solutions has also been derived by Sermange and Temam \cite{MHDleray}.
To this day, the uniqueness problem for Leray–Hopf weak solutions of the Navier–Stokes equations remains an open question in dimensions greater than two. 

By the mild formulation of equations, the global well-posedness for small data and local well-posedness for large data are derived in some critical space 
\begin{gather*}
    H^{\frac{d}{2}-1}\subset L^d\subset B^{-1+\frac{d}{p}}_{p,\infty}\subset BMO^{-1}\subset B^{-1}_{\infty,\infty}
\end{gather*}
where $1\leq p<\infty$. These norms are all invariant under the following scaling transformations:
\begin{gather*}
    u(t,x)\mapsto\lambda u(\lambda^2t,\lambda x ),\quad p(t,x)\mapsto\lambda^2p(\lambda^2t,\lambda x).
\end{gather*}
Fujita-Kato\cite{30} and Kato\cite{40} proved that this holds for initial data in the first two spaces, respectively. A corresponding result in Besov spaces $B^{-1+\frac{d}{p}}_{p,\infty}(\mathbb{R}^d)(1\leq q<\infty)$ was obtained by Cannone \cite{13} and Planchon \cite{53}. Koch-Tataru\cite{BMO-1wp} established that global well-posedness holds for small data in $BMO^{-1}$. For the MHD equations, we refer the reader to \cite{criticalMHD1,criticalMHD2,criticalMHD3} for well-posedness results on critical spaces. The weak-strong uniqueness property for (\ref{e:NS}) or (\ref{e:MHDe}) in critical spaces has also attracted considerable attention in the literature. We refer the reader to the references \cite{kozono,lady,Prodi,serrin} on NS equations and \cite{MHD1,MHD2,MHD3,MHD4} on MHD equations.

Nevertheless, ill-posedness or non-uniqueness will also arise in the critical class. For the largest critical space $B^{-1}_{\infty,\infty}$, Bourgain-Pavlovi\'c \cite{Bourgain} showed that norm inflation is possible in 3D Navier-Stokes equations. In other words, the solutions with arbitrarily small data in $B^{-1}_{\infty,\infty}$ can grow arbitrarily large in a short time. On the other hand, Germain \cite{33} proved that the data-to-solution map is not $C^2$ in the spaces $B^{-1}_{\infty,q}$ for $q>2$ and Yoneda \cite{63} showed that ill-posedness holds in this class. Later in \cite{60}, Wang showed that norm inflation from small data may occur even in the spaces $B^{-1}_{\infty,q}$ for $q\in [1,2]$, which are continuously embedded into $BMO^{-1}$. Recently in \cite{Coiculescu2025}, Coiculescu-Palasek firstly constructed two distinct global solutions with identical initial data in the space $BMO^{-1}$. 

Driven by the development of convex integration techniques \cite{introductionconvex1,introductionconvex2,Isett3,B1,2donsager}, many non-uniqueness results have been obtained in supercritical spaces. The first such result is given by Buckmaster-Vicol \cite{ns有限能量不唯一}. By the convex integration scheme and the $L^2_x$-based intermittent spatial building block, they showed non-uniqueness of finite energy weak solutions to the 3D Navier-Stokes equations. See the surveys \cite{ns有限能量不唯一低于lion指标,buckmaster1} for the fractional dissipation case. Later In \cite{serrin准则luo}, Cheskidov-Luo used the temporal intermittency method and proved the non-uniqueness of (\ref{e:NS}) on $L^p_tL^\infty_x$ when $p<2$. Moreover, Cheskidov-Luo \cite{MR4610908} established the non-uniqueness on the class $L^{\infty}_tL^q_x$ when $q<2$ for the 2D NS equation. For the 3D hyperdissipative NS equation with hyperviscosity exponent beyond the Lion exponent, Li-Qu-Zeng-Zhang \cite{qupeng} showed non-uniqueness in some supercritical spaces. The above three results are all sharp with respect to the Lady\v{z}enskaja-Prodi-Serrin criteria and posed some additional favorable properties outside of singular times with arbitrarily Hausdorff dimension. However, due to the geometry of (\ref{e:MHDe}), the above non-uniqueness results in supercritical spaces may not completely extend to the case of the MHD equations. We refer the reader to \cite{yeMHD,MHDzeng0,MHDzeng1,MHDzeng2} for further details. 

We also refer to another programme by Jia and \v{S}ver\'ak \cite{jia2,jia1}. They demonstrated that non-uniqueness of Leray–Hopf weak solutions is valid, provided that a suitable spectral condition holds for the linearized Navier–Stokes operator. Inspired by \cite{vishik1,vishik2}, Albritton-Bru\'e-Colombo \cite{forcedns} proved non-uniqueness in the Leray-Hopf weak solutions for the forced 3D Navier-Stokes equations. See also the related work \cite{abchypodissipativens} for the hypodissipative Navier-Stokes equations in two dimensions.

\subsection{Main result}
In \cite{Coiculescu2025}, the authors demonstrated that uniqueness of the 3D Navier-Stokes equation may break down if the initial data is large in $BMO^{-1}$. Inspired by their work, we establish non-uniqueness for the 5D MHD equations. More precisely, we have
\begin{theo}\label{maintheo}
    There exists divergence-free initial data $(U^0,B^0)\in BMO^{-1}$ such that the MHD equations (\ref{e:MHDe}) admits two distinct global solutions
    \begin{align*}
        (\bar{v}^{(1)}, \bar{b}^{(1)}),(\bar{v}^{(2)}, \bar{b}^{(2)})\in &C^\infty_{t,x}((0,\infty)\times\mathbb{T}^5)\cap L^\infty([0,\infty);BMO^{-1}(\mathbb{T}^5))\\
        & \cap C^0([0,\infty);\dot{W}^{-1,p}(\mathbb{T}^5)).
    \end{align*}
\end{theo}

\begin{rema}
    The $BMO^{-1}$ norm was introduced by Koch-Tataru in \cite{BMO-1wp}, and is defined by
    \begin{gather*}
        \|u\|_{BMO^{-1}}\triangleq \sup_{R>0}\sup_{x_0\in \mathbb{R}^d} \left(\int^{R^2}_0\int_{B(x_0,R)}|e^{t\Delta}u|^2dxdt
        \right)^{\frac{1}{2}}<\infty,
    \end{gather*}
    where the dimension $d\geq2$. Due to the Carleson measure characterization of $BMO^{-1}$ (see \cite{Harmonic}), we have $u\in BMO^{-1}$ if $u=\mathrm{div}f$ for some vector field $f\in (BMO)^d$.  
\end{rema}

\begin{rema}
    In \cite{BMO-1wp}, Koch-Tataru introduced a critical space $X_{KT}$ with 
    norm
    \begin{align*}
        \|u\|_{X_{KT}}\triangleq \sup_{t>0}t^{\frac{1}{2}}\|u(t)\|_{L^\infty}+\sup_{x\in\mathbb{R}^3,R>0} R^{-\frac{3}{2}}\left(\int^{R^2}_0\int_{B(x,R)}|u(t,y)|^2dydt\right)^{\frac{1}{2}}.
    \end{align*}
    By verifying that the Navier-Stokes bilinear operator are bounded on $X_{KT}$, the authors implemented Kato's method to obtain the global well-posedness for small initial data in $BMO^{-1}$. The solutions $(\bar{v}^{(1)}, \bar{b}^{(1)})$ and $(\bar{v}^{(2)}, \bar{b}^{(2)})$ we constructed in Theorem \ref{maintheo} also belong to the same class. Moreover, it is rather easy to get that the initial data $(U^0,B^0)$ is smooth outside of a finite set $\{x_\xi\}_{\xi\in\Lambda_U\cup\Lambda_B}$ (see more detail in Section \ref{subsectionMikado}).
\end{rema}

\begin{rema}
    Theorem \ref{maintheo} also reveals the phenomenon of non-uniqueness of MHD equations when the dimension is greater than 5. In fact, the solution pairs in Theorem \ref{maintheo} satisfy the higher-dimensional MHD equations when zero components are added to the remaining dimensions. Owing to the geometric structure of system (\ref {e:MHDe}) and the properties of Mikado flows, the approach employed herein appears to fail to apply in dimensions lower than 5 (see details in Section \ref{subsectionMikado}). The difficluties we face are similar to those encountered for the 2D Navier–Stokes equations and those faced in the approach to Onsager's conjecture for the 2D Euler equations. In our view, once the non-uniqueness problem for the the 2D Navier–Stokes equations with initial data in $BMO^{-1}$ will be settled, the corresponding problem for the MHD equations is expected to be solved.
\end{rema}

\subsection{Main idea of the construction}
The non-uniqueness mechanism of MHD equations follows from the approach of \cite{Palasek2025} and \cite{Coiculescu2025}. In this part, we outline the main idea of the proof, omitting certain details for the sake of exposition. In our construction, there exist two types of flows:

\noindent
$\bullet$ The first is the heat-dominated flow, which is the solution of heat equaiton, up to a small error. 
\begin{align*}
    \begin{cases}
        \partial_tv_k-\Delta v_k=l.o.t,\\
        \partial_tv_k^B-\Delta v_k^B=l.o.t,\\
        \partial_tb_k-\Delta b_k=l.o.t,
    \end{cases}
\end{align*}
where $l.o.t$ means some lower order terms. Moreover, we let 
\begin{align*}
    \mathrm{supp}\,v_k\cap\mathrm{supp}\,v_k^B=\mathrm{supp}\,v_k\cap\mathrm{supp}\,b_k=\emptyset.
\end{align*}
And the heat-dominated flows satisfy the following exponential decay in time:
\begin{gather*}
    \|(v_k,v_k^B,b_k)\|_{L^\infty(\mathbb{T}^5)}\lesssim N_k e^{-N_{k}^2t}.
\end{gather*}

\noindent
$\bullet$ The second is the inverse cascade-dominated flow, whose temporal partial derivative annihilates the nonlinear self-interactions of the heat-dominated flow.
\begin{align*}
    \begin{cases}
        \partial_t\bar{v}_k+\mathbb{P}\mathrm{div}(v_{k+1}\otimes v_{k+1}+v_{k+1}^B\otimes v_{k+1}^B-b_{k+1}\otimes b_{k+1})=l.o.t,\\
         \partial_t\bar{b}_k+\mathrm{div}(v_{k+1}^B\otimes b_{k+1}-b_{k+1}\otimes v_{k+1}^B)=l.o.t.\\
         \bar{v}_k(0,\cdot)=v_k(0,\cdot)+v_k^B(0,\cdot),\\
         \bar{b}_k(0,\cdot)=b_k(0,\cdot).
    \end{cases}
\end{align*}
Compared to the heat-dominated flows, the inverse cascade-dominated flows have slower exponential decay in time:
\begin{gather*}
    \|(\bar{v}_k,\bar{b}_k)\|_{L^\infty(\mathbb{T}^5)}\lesssim N_k e^{-N_{k+1}^2t}.
\end{gather*}

The principal part of the solutions to the magnetohydrodynamic equations is derived and given as follows:
\begin{gather*}
    v^{(1)}\triangleq\sum_{k\geq0\,even}(v_k+v_k^B)+\sum_{k\geq0\,odd}\bar{v}_k,\quad b^{(1)}\triangleq\sum_{k\geq0\,even}b_k+\sum_{k\geq0\,odd}\bar{b}_k
    \end{gather*}
    and 
    \begin{gather*}
    v^{(2)}\triangleq\sum_{k\geq0\,odd}(v_k+v_k^B)+\sum_{k\geq0\,even}\bar{v}_k,\quad b^{(2)}\triangleq\sum_{k\geq0\,odd}b_k+\sum_{k\geq0\,even}\bar{b}_k.
\end{gather*}

In the following discussion, $i$ denotes either $1$ or $2$. As a result of the cancellation between the heat-dominated flows and the inverse cascade-dominated flows, we can deduce remaining terms $(F^{(i)},G^{(i)})$ which are defined by
\begin{align*}
    \begin{cases}
        (\partial_t-\Delta) v^{(i)}+ \mathbb{P}\text{div}( v^{(i)}\otimes  v^{(i)}-b^{(i)}\otimes b^{(i)})=-\mathbb{P}\mathrm{div}F^{(i)},  \\
        (\partial_t-\Delta) b^{(i)}+\mathrm{div}(v^{(i)}\otimes  b^{(i)}-b^{(i)}\otimes v^{(i)})=-\mathrm{div}G^{(i)},
    \end{cases}
\end{align*}
are arbitrarily small in some subcritial norms. In other words, we conclude that $(v^{(1)},b^{(1)})$ and $(v^{(2)},b^{(2)})$ are solutions to the MHD equations (\ref{e:MHDe}) with a small error in divergence form. 

In the final step, we add the perturbation $(\omega^{(i)},\rho^{(i)})$ to $(v^{(i)},b^{(i)})$ to correct the associated errors, so that $(\bar{v}^{(i)},\bar{b}^{(i)})\triangleq(v^{(i)}+\omega^{(i)},b^{(i)}+\rho^{(i)})$ is indeed a solution to the MHD equations (\ref{e:MHDe}). Then $(\omega^{(i)},\rho^{(i)})$ must satisfy the perturbed MHD equations
\begin{align*}
        \begin{cases}
            (\partial_{t}-\Delta)\omega^{(i)}+\mathbb{P}\text{div}\, (v^{(i)}\otimes \omega^{(i)}+\omega^{(i)}\otimes v^{(i)}-b^{(i)}\otimes \rho^{(i)}- \rho^{(i)}\otimes b^{(i)}+\omega^{(i)}\otimes \omega^{(i)}-\rho^{(i)} \otimes \rho^{(i)})\\\,=\mathbb{P}\text{div}\, F^{(i)},\\
            (\partial_{t}-\Delta)\rho^{(i)}+\text{div}\,(v^{(i)}\otimes \rho^{(i)}-\rho^{(i)}\otimes v^{(i)}+\omega^{(i)}\otimes b^{(i)}-b^{(i)}\otimes \omega^{(i)}+\omega^{(i)}\otimes \rho^{(i)}-\rho^{(i)}\otimes \omega^{(i)})\\\,=\text{div}\, G^{(i)},\\
            \rho^{(i)}(0,\cdot)=\omega^{(i)}(0,\cdot)=0.
        \end{cases}
\end{align*}
Given that $(F^{(i)},G^{(i)})$ is small in certain subcritical spaces, we are able to establish the existence of $(\omega^{(i)},\rho^{(i)})$ and confirm its smallness in appropriate subcritical norms by means of a fixed point argument. 

Using the fact that $\bar{v}_k(0,\cdot)=v_k(0,\cdot)+v_k^B(0,\cdot)$, $\bar{b}_k(0,\cdot)=b_k(0,\cdot)$ and $\rho^{(i)}(0,\cdot)=w^{(i)}(0,\cdot)=0$, we know that $(\bar{v}^{(1)},\bar{b}^{(1)})$ and $(\bar{v}^{(2)},\bar{b}^{(2)})$ have the same initial data in $BMO^{-1}$, which is ensured by some support estimates. However, owing to the distinct temporal decay behaviors of $(v_k,v_k^B,b_k)$ and $(\bar{v}_k,\bar{b}_k)$, we can infer that $(\bar{v}^{(1)},\bar{b}^{(1)})$ and $(\bar{v}^{(2)},\bar{b}^{(2)})$ are distinct (see details in Section \ref{Proof of the Non-Uniqueness}).

\subsection{Notation and parameters}
In this section, we collect the notation used throughout the paper and introduce the parameters appearing in our construction.

Let $S^{5\times5}(\mathbb{R})$ be the space of real symmetric $5\times5$ matrices and $A^{5\times5}(\mathbb{R})$ be the space of real antisymmetric $5\times5$ matrices. For any real numbers $a$ and $b$, we write $a\wedge b$ and $a\vee b$ to denote $\min\{a,b\}$ and $\max\{a,b\}$, respectively. 

We use the notation $X\lesssim Y$ to mean that $X\leq CY$ with some constants $C>0$ independent of A, but which may depend on other parameters. We write $\lfloor x \rfloor$ means the largest integer $\geq x$.

We introduce a parameter $\alpha\in(0,\frac{1}{8})$ to quantify the subcriticality of the remaining terms $(F^{(i)},G^{(i)})$ and the perturbations $(\omega^{(i)},\rho^{(i)})$. To measure the amplitude and the frequency of the building blocks, we introduce the parameters
\begin{align*}
M_k\triangleq\lfloor A^{b^k} \rfloor, \quad
N_k \triangleq M_k \lfloor M_k^{\gamma-1}\rfloor, \quad \mathrm{for}\,\, k\in\mathbb{N},
\end{align*}
where $b>10\vee\frac{1}{1-8\alpha}$ and $\frac{1}{1-4\alpha}<\gamma<\frac{1}{4\alpha+b^{-1}}$. Moreover, we introduce the small H\"{o}lder exponent $\kappa\in(0,(\frac{1}{2}-\frac{1}{2\gamma}-2\alpha)\wedge\alpha)$. Finally, $A\gg1$ depending on other parameters will be chosen sufficiently large at the end.

\subsection{Organization of the paper}
The organization of the rest of the paper is as follows. In Section \ref{sectprincipal}, we design the intial data and the leading parts of the solutions. In Section \ref{sec:4}, we verify that the remaining terms are small in some subcritical norms and estimate the perturbations. In Section \ref{Proof of the Non-Uniqueness}, we complete the proof of Theorem \ref{maintheo} to show the non-uniqueness. Appendix contains some technical tools used in the paper.

\section{The leading part of the solutions}\label{sectprincipal}

In this section, we mainly construct the crucial part of the solutions. The key building block is the Mikado flow, which is introduced by \cite{6}, garantees that initial data belongs to $BMO^{-1}$.

\subsection{Construction of the Mikado potential}\label{subsectionMikado}

Let $\varphi\in C^\infty_0(B(0,1),\mathbb{R})$ satisfying $\varphi(0)=1$, where $B(0,1)$ is the ball of radius 1 centered at 0 in $\mathbb{R}^{d-2}$. By an abuse of nation, we periodize $\varphi$ so that $\varphi$ is treated as periodic functions defined on $\mathbb{T}^{d-2}$. For any $\xi\in \Lambda_U\cup \Lambda_B$ and $k\geq0$, we define Mikado flow as follows,
\begin{align}
    \varphi_{\xi}(x)&\triangleq\varphi(\xi x,\xi_3x,\xi_{4}x), \quad\forall x\in\mathbb{T}^5,\label{defivarphixi}
    \\
    \varphi_{\xi,k}(x)&\triangleq\varphi_{\xi}(M_k\xi(x-x_{\xi})), \quad\forall x\in\mathbb{T}^5,\label{defi Mikado}
\end{align}
where $x_\xi\in \mathbb{T}^5$.
Since the dimension is 5, we can choose $\{x_\xi\}_{\Lambda_U\cup \Lambda_B}$ such that $\{\xi_1\mathbb{R}+\xi_2\mathbb{R}+x_\xi\, \,mod\,(2\pi\mathbb{Z})^5\}_{\Lambda_U\cup \Lambda_B}$ never intersect each other. 
More precisely, we can achieve this by observing that for any distinct $\xi,\xi'\in {\Lambda_U\cup \Lambda_B}$, these $4\times5$ nonhomogeneous linear systems
\begin{align*}
        y_1\xi_1+y_2\xi_2-y_1'\xi_1'+y_2'\xi_2'=x_{\xi'}-x_{\xi}
\end{align*}
are unsolvable for $(y_1,y_2,y_1',y'_2)\in \mathbb{R}^4$ if we choose suitable $\{x_\xi\}_{\Lambda_U\cup \Lambda_B}$ carefully.

Note that 
\begin{gather}
    \text{supp}\,\varphi_{\xi,k}=\{\xi_1\mathbb{R}+\xi_2\mathbb{R}+x_\xi+\text{B}_\xi(0,\frac{1}{M_k})\,\, mod\,(2\pi\mathbb{Z})^5\},\quad\forall \xi\in\Lambda_U\cup \Lambda_B,\label{supp defi}
\end{gather}
where $\text{B}_\xi(0,\frac{1}{M_k})$ denotes $\{\xi(-\frac{1}{M_k},\frac{1}{M_k})+\xi_3(-\frac{1}{M_k},\frac{1}{M_k})+\xi_{4}(-\frac{1}{M_k},\frac{1}{M_k})\}$ in $\mathbb{R}^5$. Since $\{\xi_1\mathbb{R}+\xi_2\mathbb{R}+x_\xi\, \,mod\,(2\pi\mathbb{Z})^5\}_{\Lambda_U\cup \Lambda_B}$ never intersect each other, we choose A sufficiently large such that $\text{supp}\,\varphi_{\xi,k}\cap\text{supp}\,\varphi_{\xi',k}=\emptyset$ for any distinct $\xi, \xi'\in \Lambda_U\cup \Lambda_B$. 

Further, we define slight expansions of the support regions of $\{\varphi_{\xi,k}\}$: For $j\in\{1,2,3\}$, $\Omega_k^{j}$ is j$N_{k+1}^{-\frac{1}{3}}N_{k}^{-\frac{2}{3}}$-neighborhood of $\cup_{\xi\in\Lambda_U\cup \Lambda_B }\text{supp}\,\varphi_{\xi,k}$.

In order to estimate initial data in the $BMO^{-1}$ norm, we introduce the Mikado flows in term of vector potentials (magnetic potentials) rather than velocity fields (magnetic fields).
\begin{defi}\label{defi potential}
    We define Mikado flow potentials:
    \begin{align*}
        \Psi^0_{\xi,k}(x)&\triangleq N_k^{-2}\varphi_{\xi,k}(x)\mathrm{sin}(N_k(x-x_\xi)\cdot \xi)\xi_1,\quad\forall\xi\in\Lambda_U,\\
        \Phi^0_{\xi,k}(x)&\triangleq N_k^{-2}\varphi_{\xi,k}(x)\mathrm{sin}(N_k(x-x_\xi)\cdot \xi)\xi_1,\quad\forall\xi\in\Lambda_B.\\
        \Theta^0_{\xi,k}(x)&\triangleq N_k^{-2}\varphi_{\xi,k}(x)\mathrm{sin}(N_k(x-x_\xi)\cdot \xi)\xi_2,\quad\forall\xi\in\Lambda_B.
    \end{align*}
\end{defi}
    
Lemma \ref{lemma divfree} below contains some basic properties of the Mikado flow potentials.
\begin{lemm}\label{lemma divfree} For $k\in\mathbb{N}$, we have the following properties:

\noindent
1.With $\{\varphi_{\xi,k}\}$ defined in (\ref{defi Mikado}), we have 
\begin{align*}
    \text{supp}\,\varphi_{\xi,k}\cap\text{supp}\,\varphi_{\xi',k}=\emptyset,\quad\forall \xi\neq\xi'\in\Lambda_U\cap\Lambda_B,
\end{align*}
and 
\begin{align}
    \|\nabla^m\varphi_{\xi,k}\|_{L^{\infty}(\mathbb{T}^5)}\lesssim M_k^m,\quad \forall \xi\in\Lambda_U\cap\Lambda_B,\quad\forall m\in\mathbb{N}.\label{e varphi}
\end{align}

\noindent 
2.Each $\Psi_{\xi,k}^0$, $\Phi_{\xi,k}^0$ and $\Theta_{\xi,k}^0$ are divergence-free and solve the following steady equations:
\begin{gather}
    \mathrm{div}\,(\Psi_{\xi,k}^0\otimes \Psi_{\xi,k}^0)=0,\quad\forall\xi\in\Lambda_U,\\
    \mathrm{div}\,(\Phi_{\xi,k}^0\otimes \Phi_{\xi,k}^0)=\mathrm{div}\,(\Theta_{\xi,k}^0\otimes \Theta_{\xi,k}^0)=0,\quad\forall\xi\in \Lambda_B,\\
    \mathrm{div}\,(\Phi_{\xi,k}^0\otimes \Theta_{\xi,k}^0)=\mathrm{div}\,(\Theta_{\xi,k}^0\otimes \Phi_{\xi,k}^0)=0,\quad\forall\xi\in\Lambda_B.
\end{gather}
Moreover, we have
\begin{gather}
    \|\nabla^m\Psi^0_{\xi,k}\|_{L^{\infty}(\mathbb{T}^5)}\lesssim N_k^{-2+m},\quad\forall\xi\in\Lambda_U,\label{e Psi}\\
    \|\nabla^m\Phi^0_{\xi,k}\|_{L^{\infty}(\mathbb{T}^5)}\lesssim N_k^{-2+m},\quad\forall\xi\in\Lambda_B,\label{e Phi}\\
    \|\nabla^m\Theta^0_{\xi,k}\|_{L^{\infty}(\mathbb{T}^5)}\lesssim N_k^{-2+m},\quad\forall\xi\in\Lambda_B.\label{e Theta}
\end{gather}

\noindent
3.With 
\begin{gather*}
    A_{\xi,k}\triangleq\fint_{\mathbb{T}^5}\varphi_{\xi,k}^2(x)\mathrm{sin}^2(N_k(x-x_\xi)\cdot \xi)dx,\quad\forall\xi\in\Lambda_U,\\
    B_{\xi,k}\triangleq\fint_{\mathbb{T}^5}\varphi_{\xi,k}^2(x)\mathrm{sin}^2(N_k(x-x_\xi)\cdot \xi)dx,\quad\forall\xi\in \Lambda_B,
\end{gather*}
there exists an universal constant $\delta>0$ such that
\begin{gather}
    \delta^{-1}\leq A_{\xi,k}\leq \delta,\quad\forall\xi\in\Lambda_U,\label{e Axi}\\
    \delta^{-1}\leq B_{\xi,k}\leq \delta,\quad\forall\xi\in\Lambda_B.\label{e Bxi}
\end{gather}

\noindent
4.For any subset $Q\subset\mathbb{T}^5$, we have
\begin{gather}
    |\Omega_k^3\cap Q|\leq2^{k(Q)-k}|Q|,\quad\forall k\geq k(Q),\label{e Omegak}
\end{gather}
where $k(Q)\triangleq \inf\{k\in\mathbb{N}:|\Omega_k^3|\leq|Q|\}$.
\end{lemm}

\begin{proof}
    The proof of properties 1-3 is trivial and thus is omitted. In the following, we verify property 4 directly. When $k=k(Q)$, it is clearly valid. For any $k> k(Q)$, by (\ref{supp defi}) and the definition of $k(Q)$, we have
    \begin{align*}
        |\Omega_k^3\cap Q|&\leq|\Omega_k^3|\\
        &=C_d(M_k^{-1}+3N_{k+1}^{-\frac{1}{3}}N_{k}^{-\frac{2}{3}})^{3}\\
        &\leq C_d(M_{k(Q)}^{-1}+3N_{{k(Q)}+1}^{-\frac{1}{3}}N_{{k(Q)}}^{-\frac{2}{3}})^{3}(\frac{M_{k(Q)}}{M_k}+(\frac{N_{k(Q)+1}}{N_{k+1}})^{\frac{1}{3}}(\frac{N_{k(Q)}}{N_k})^{\frac{2}{3}})^{3}\\
        &\leq(2\frac{M_{k(Q)}}{M_k})^{3}|Q|,
    \end{align*}
    where the last inequality follows from the fact 
    \begin{gather*}
        (\frac{N_{k(Q)+1}}{N_{k+1}})^{\frac{1}{3}}(\frac{N_{k(Q)}}{N_k})^{\frac{2}{3}}\leq (\frac{M_{k(Q)}}{M_k})^\gamma\leq\frac{M_{k(Q)}}{M_k},\quad\forall k>k(Q).
    \end{gather*}
   
    Moreover, by the rapid growth of $M_k$ and choosing A sufficiently large, it follows immediately that
    \begin{align*}
        (2\frac{M_{k(Q)}}{M_k})^{3}<2^{k(Q)-k}, \quad\forall k>k(Q),
    \end{align*}
    which completes the proof of property 4.
\end{proof}

\subsection{Construction of the initial data}
Based on the Mikado flow potentials defined in above, we build the vector (magnetic) potentials for the initial data and then the initial data in this part. Firstly, Let us introducing some useful elements: 
\begin{defi}
    We define:
\begin{itemize}
    \item Two types of differrential operators that take vector fileds and magnetic fileds to symmetric rank-2 tensors and skew-symmetric rank-2 tensors respectively,
\begin{align*}
    \mathcal{D}f&\triangleq-\nabla f-\nabla f^{T}+2(\mathrm{div}f)\mathrm{Id},\\
    \widetilde{\mathcal{D}}f&\triangleq-\nabla f+\nabla f^{T}.
\end{align*}
\item two cutoff functions $\chi_k^V$ and $\chi_k^B$. $\chi_k^V$ is identically 1 in $\Omega_{k-1}^2$ and identically 0 outside of $\Omega_{k-1}^3$. $\chi_k^B$ is identically 1 in $\Omega_{k-1}^1$ and identically 0 outside of $\Omega_{k-1}^2$. Moreover, 
\begin{gather}
    \|\nabla^m(\chi_k^V,\chi_k^B)\|_{L^\infty(\mathbb{T}^5)}\lesssim(N_k^{\frac{1}{3}}N_{k-1}^{\frac{2}{3}})^m.\label{e chi}
\end{gather}
\item A standard mollifier $\eta_k$ with length scale $N_{k+1}^{-\frac{1}{3}}N_{k}^{-\frac{2}{3}}$, supported in $B(0,1)$ with $\int_{\mathbb{T}^5}\phi_k=1$.
\end{itemize}
\end{defi}

We bulid the vector potentials and the magnetic potentials by induction. Recalling $\varepsilon_0>0$ is the universal constant introduced in Lemma \ref{lemmaNS} and \ref{lemmaMHD}.
\begin{defi}\label{defi of psiok}
    Let $\bar{\xi}$ and $\bar{\bar{\xi}}$ be two vectors in $\Lambda_U$ and $\Lambda_B$, respectively. We define  
    \begin{gather*}
        \psi^0_0\triangleq N_0\eta_0*\Psi^0_{\bar{\xi},0},\\
        \phi^0_0\triangleq N_0\eta_0*\Phi^0_{\bar{\bar{\xi}},0},\\
        \theta^0_0\triangleq N_0\eta_0*\Theta^0_{\bar{\bar{\xi}},0}.
    \end{gather*}
    Given $\psi^0_{k-1}$, $\phi^0_{k-1}$ and $\theta^0_{k-1}$, we define   
    \begin{gather}
    a_{\xi,k}(x)=N_k(\frac{2l_k}{A_{\xi,k}})^{\frac{1}{2}}\chi_k^V(x)\Gamma_\xi(\mathrm{Id}+l_k^{-1}(\mathcal{D}\psi^0_{k-1}+\mathcal{D}\theta^0_{k-1}-s_{k})),\quad\forall\xi\in\Lambda_U,\label{defi of a}\\
    b_{\xi,k}(x)=N_k(\frac{2\|\widetilde{D}\phi_{k-1}^0\|_{L^\infty(\mathbb{T}^5)}}{\varepsilon_0B_{\xi,k}})^{\frac{1}{2}}\chi_k^B(x)\gamma_{\xi}(\frac{\varepsilon_0\widetilde{D}\phi_{k-1}^0}{\|\widetilde{D}\phi_{k-1}^0\|_{L^\infty(\mathbb{T}^5)}}),\quad\forall\xi\in\Lambda_B,\label{defi of b}
    \end{gather}
    where
    \begin{gather}
    l_k\triangleq\varepsilon_0^{-1}(\|\mathcal{D}\psi^0_{k-1}\|_{L^\infty(\mathbb{T}^5)}+\|\mathcal{D}\theta^0_{k-1}\|_{L^\infty(\mathbb{T}^5)}+\|s_k\|_{L^\infty(\mathbb{T}^5)}),\\
    s_k\triangleq(2N_k)^{-2}\sum_{\xi\in\Lambda_B}B_{\xi,k}b_{\xi,k}^2(\xi_2\otimes\xi_2-\xi_1\otimes\xi_1).\label{defi Sk}
\end{gather}
We inductively define
    \begin{align}
\psi^0_k\triangleq\eta_k*\sum_{\xi\in\Lambda_U}a_{\xi,k}\Psi^0_{\xi,k},\\
\phi^0_k\triangleq\eta_k*\sum_{\xi\in\Lambda_B}b_{\xi,k}\Phi^0_{\xi,k},\label{definition of phi}\\
\theta^0_k\triangleq\eta_k*\sum_{\xi\in\Lambda_B}b_{\xi,k}\Theta^0_{\xi,k}.
    \end{align}

\end{defi}

From Definition \ref{defi of psiok}, we get the following estimation.
\begin{lemm}
For any $k\geq1$, it holds that
    \begin{gather}
        \|s_k\|_{L^{\infty}(\mathbb{T}^5)}\lesssim\|\widetilde{D}\phi_{k-1}^0\|_{L^\infty(\mathbb{T}^5)},\label{e S}\\
        \|a_{\xi,k}\|_{L^{\infty}(\mathbb{T}^5)}\lesssim N_k(\|\mathcal{D}\psi^0_{k-1}\|_{L^\infty(\mathbb{T}^5)}+\|\mathcal{D}\theta^0_{k-1}\|_{L^\infty(\mathbb{T}^5)}+\|\widetilde{D}\phi_{k-1}^0\|_{L^\infty(\mathbb{T}^5)})^{\frac{1}{2}},\label{e a}\\
         \|b_{\xi,k}\|_{L^{\infty}(\mathbb{T}^5)}\lesssim N_k(\|\widetilde{D}\phi_{k-1}^0\|_{L^\infty(\mathbb{T}^5)})^{\frac{1}{2}},\label{e b}
    \end{gather}
    and
    \begin{gather}
        \|\nabla s_k\|_{L^{\infty}(\mathbb{T}^5)}\lesssim  N_k^{\frac{1}{3}}N_{k-1}^{\frac{2}{3}}\|\widetilde{D}\phi_{k-1}^0\|_{L^\infty(\mathbb{T}^5)},\\
        \|\nabla a_{\xi,k}\|_{L^{\infty}(\mathbb{T}^5)}\lesssim N_k^{\frac{4}{3}}N_{k-1}^{\frac{2}{3}}(\|\mathcal{D}\psi^0_{k-1}\|_{L^\infty(\mathbb{T}^5)}+\|\mathcal{D}\theta^0_{k-1}\|_{L^\infty(\mathbb{T}^5)}+\|\widetilde{D}\phi_{k-1}^0\|_{L^\infty(\mathbb{T}^5)})^{\frac{1}{2}},\label{e nabla a}\\
        \|\nabla b_{\xi,k}\|_{L^{\infty}(\mathbb{T}^5)}\lesssim N_k^{\frac{4}{3}}N_{k-1}^{\frac{2}{3}}(\|\widetilde{D}\phi_{k-1}^0\|_{L^\infty(\mathbb{T}^5)})^{\frac{1}{2}}.\label{e nabla b}
    \end{gather}
\end{lemm}
\begin{proof}
    Combining (\ref{e Bxi}), (\ref{defi of b}) and (\ref{gamma jie}), it is rather easy to get (\ref{e S}) and (\ref{e b}). Similarly, (\ref{e Axi}), (\ref{defi of a}), (\ref{e S}) and (\ref{Gamma xiajie}) imply (\ref{e a}).
  
   With a standard mollifier estimate , we have
   \begin{gather}
       \|\nabla\widetilde{D}\phi_{k-1}^0\|_{L^\infty(\mathbb{T}^5)}\lesssim N_k^{\frac{1}{3}}N_{k-1}^{\frac{2}{3}}\|\widetilde{D}\phi_{k-1}^0\|_{L^\infty(\mathbb{T}^5)}. \label{e mollification}
   \end{gather}
   Thus, combining (\ref{e Bxi}), (\ref{e chi}), (\ref{defi of b}) and (\ref{e mollification}), we obtain
   \begin{gather*}
       \|\nabla b_{\xi,k}\|_{L^{\infty}(\mathbb{T}^5)}\lesssim N_k(\|\widetilde{D}\phi_{k-1}^0\|_{L^\infty(\mathbb{T}^5)})^{\frac{1}{2}}N_k^{\frac{1}{3}}N_{k-1}^{\frac{2}{3}}\lesssim N_k^{\frac{4}{3}}N_{k-1}^{\frac{2}{3}}(\|\widetilde{D}\phi_{k-1}^0\|_{L^\infty(\mathbb{T}^5)})^{\frac{1}{2}}.
   \end{gather*}
   Moreover, due to (\ref{e Bxi}), (\ref{defi Sk}) and (\ref{e nabla b}), we have
   \begin{gather*}
        \|\nabla s_k\|_{L^{\infty}(\mathbb{T}^5)}\lesssim  N_k^{\frac{1}{3}}N_{k-1}^{\frac{2}{3}}\|\widetilde{D}\phi_{k-1}^0\|_{L^\infty(\mathbb{T}^5)}.
   \end{gather*}
   
   The proof of (\ref{e nabla a}) is similar, here we omit the detail.
\end{proof}

Because $\|\widetilde{\mathcal{D}}\phi^0_{k-1}\|_{L^{\infty}}$, $\|{\mathcal{D}}\theta^0_{k-1}\|_{L^{\infty}}$ and $\|{\mathcal{D}}\psi^0_{k-1}\|_{L^{\infty}}$ appear in the denominator, it is necessary to estimate their lower bounds. In fact, the following proposition shows that their magnitudes are essentially constant.
\begin{prop}\label{prop of Dpsi}
    There exists a universal constant $C$ such that the following estimates hold
    \begin{gather}
        C^{-2}(C^2\|\widetilde{\mathcal{D}}\phi^0_0\|_{L^{\infty}})^{2^{-k}}\leq \|\widetilde{\mathcal{D}}\phi^0_k\|_{L^{\infty}}\leq C^{2}(C^2\|\widetilde{\mathcal{D}}\phi^0_0\|_{L^{\infty}})^{2^{-k}},\quad\forall k\in\mathbb{N},\label{estimate Dphi}\\
        C^{-2}(C^2\|\widetilde{\mathcal{D}}\phi^0_0\|_{L^{\infty}})^{2^{-k}}\leq \|{\mathcal{D}}\theta^0_k\|_{L^{\infty}}\leq C^{2}(C^2\|\widetilde{\mathcal{D}}\phi^0_0\|_{L^{\infty}})^{2^{-k}},\quad\forall k\in\mathbb{N},\label{e Dtheta}\\
        C^{-2}(C^2\|{\mathcal{D}}\psi^0_0\|_{L^{\infty}})^{2^{-k}}\leq \|{\mathcal{D}}\psi^0_k\|_{L^{\infty}}\leq C^{2}(C^2(\|{\mathcal{D}}\psi^0_0\|_{L^{\infty}}+\|{\widetilde{\mathcal{D}}}\phi^0_0\|_{L^{\infty}}))^{2^{-k}},\quad\forall k\in\mathbb{N}.\label{estimate Dpsi}
    \end{gather}
\end{prop}

\begin{proof}
    The proof of (\ref{estimate Dphi}) and (\ref{e Dtheta}) is similar to \cite[Proposition 3.7]{Coiculescu2025}, here we omit it. In the following, we establish (\ref{estimate Dpsi}) by induction. Noting that (\ref{estimate Dpsi}) is clear for $k=0$. We suppose that (\ref{estimate Dpsi}) holds for $k-1$. We divide $\|{\mathcal{D}}\psi^0_k\|_{L^{\infty}}$ into 3 parts: \textit{I} when $\mathcal{D}$ falls on $\mathrm{sin}$ within $\Psi_{\xi,k}^0$; \textit{II} when it falls on $\varphi_{\xi,k}$ within $\Psi_{\xi,k}^0$; \textit{III} when it falls on $a_{\xi,k}$. Moreover, we proceed to decompose the term $\textit{I}$ into two distinct components:
    \begin{gather}
        I_1=N_k^{-2}\sum_{\xi\in\Lambda_U}\varphi_{\xi,k}(x)a_{\xi,k}(x)\mathcal{D}(\mathrm{sin}(N_k(x-x_\xi)\cdot \xi)\xi_1),\label{defi I1}\\
        I_2=\eta_k*I_1-I_1.
    \end{gather}
    Using (\ref{defi of a}) and the fact that $\{\varphi_{\xi,k}\}$ have disjoint support, we get
    \begin{align*}
        &\quad\, I_1|_{\mathrm{supp}\varphi_{\xi,k}}\\&=-\varphi_{\xi,k}(x)(\frac{2l_k}{A_{\xi,k}})^{\frac{1}{2}}\chi_k(x)\Gamma_\xi(\mathrm{Id}+l_k^{-1}(\mathcal{D}\psi^0_{k-1}+\mathcal{D}\theta^0_{k-1}-s_{k})\mathrm{cos}(N_k(x-x_\xi)\cdot \xi)(\xi\otimes\xi_1+\xi_1\otimes\xi).
    \end{align*}
    Since $\varphi_{\xi,k}(x_{\xi})=\chi_k(x_\xi)=\mathrm{cos}(N_k(x_{\xi}-x_\xi)\cdot \xi)=1$, (\ref{Gamma xiajie}) and induction hypothesis, by choosing sufficiently large $C$, we obtain
    \begin{gather}
        \|I_1\|_{L^\infty(\mathbb{T}^5)}\gtrsim l_k\gtrsim(\|\mathcal{D}\psi^0_{k-1}\|_{L^\infty(\mathbb{T}^5)})^{\frac{1}{2}}\geq100C^{-2}(C^2\|{\mathcal{D}}\psi^0_0\|_{L^{\infty}})^{2^{-k}}.\label{e I1low}
    \end{gather}
    
    On the other hand, combining (\ref{e a}), (\ref{defi I1}) and the induction hypothesis, we achieve
    \begin{align}
        \|I_1\|_{L^\infty(\mathbb{T}^5)}&\lesssim(\|\mathcal{D}\psi^0_{k-1}\|_{L^\infty(\mathbb{T}^5)}+\|\mathcal{D}\theta^0_{k-1}\|_{L^\infty(\mathbb{T}^5)}+\|\widetilde{D}\phi_{k-1}^0\|_{L^\infty(\mathbb{T}^5)})^{\frac{1}{2}}\label{e I1upper}\\&\lesssim \frac{1}{2} C^{2}(C^2(\|{\mathcal{D}}\psi^0_0\|_{L^{\infty}}+\|{\widetilde{\mathcal{D}}}\phi^0_0\|_{L^{\infty}}))^{2^{-k}}.\nonumber
    \end{align}
    And (\ref{e varphi}), (\ref{e nabla a}), (\ref{defi I1}) and the induction hypothesis imply
    \begin{gather}
        \|\nabla I_1\|_{L^\infty(\mathbb{T}^5)}\lesssim(\|\mathcal{D}\psi^0_{k-1}\|_{L^\infty(\mathbb{T}^5)}+\|\mathcal{D}\theta^0_{k-1}\|_{L^\infty(\mathbb{T}^5)}+\|\widetilde{D}\phi_{k-1}^0\|_{L^\infty(\mathbb{T}^5)})^{\frac{1}{2}}(M_k+N_k^{\frac{1}{3}}N_{k-1}^{\frac{2}{3}}+N_k)\lesssim N_k.
    \end{gather}
    Here, we choose $C$ to be sufficiently large.
    Thus, by a standard mollifier estimate, we get
    \begin{gather}
        \|I_2\|_{L^\infty(\mathbb{T}^5)}\lesssim \|\nabla I_1\|_{L^\infty(\mathbb{T}^5)}N_{k+1}^{-\frac{1}{3}}N_k^{-\frac{2}{3}}\lesssim(\frac{N_k}{N_{k+1}})^{\frac{1}{3}}.\label{e I2}
    \end{gather}
    
    Next, we consider \textit{II} and \textit{III}. From (\ref{e varphi}), (\ref{e a}), (\ref{e nabla a}) and the induction hypothesis, we can conclude that
    \begin{gather}
        \| \textit{II}\,\|_{L^\infty(\mathbb{T}^5)}\lesssim \frac{M_k}{N_k},\label{e II}\\
        \| \textit{III}\,\|_{L^\infty(\mathbb{T}^5)}\lesssim (\frac{N_{k-1}}{N_k})^{\frac{2}{3}}.\label{e III}
    \end{gather}

    Finally, combining (\ref{e I1low}), (\ref{e I1upper}), (\ref{e I2}), (\ref{e II}) and (\ref{e III}) together, we obtain (\ref{estimate Dpsi}) by choosing $A$ sufficiently large.
\end{proof}

\begin{prop} \label{prop nablampsi}
   For any $m\geq0$, it holds
   \begin{gather}
       \|\nabla^m(a_{\xi,k},b_{\xi',k})\|_{L^{\infty}(\mathbb{T}^5)}\lesssim_m N^m_{k-1}N_k,\quad\forall (\xi,\xi')\in\Lambda_U\times\Lambda_B,\quad\forall k\geq1,\label{e nablbma}\\
        \|\nabla^m(\psi^0_k,\phi^0_k,\theta^0_k)\|_{L^{\infty}(\mathbb{T}^5)}\lesssim_m N_k^{-1+m},\quad\forall k\geq 0.\label{e nablampsi}
    \end{gather}
\end{prop}

\begin{proof}
    (\ref{defi of a}), (\ref{defi of b}), Proposition \ref{prop of Dpsi} and the Leibnitz rule imply the above results.
\end{proof}

\begin{prop}
Given $\psi^0_k$, $\phi^0_k$ and $\theta^0_k$ in Definition \ref{defi of psiok}, we have
    \begin{gather}
        \sum_{k=0}^\infty\mathrm{curl}(\psi^0_k+\theta^0_k)\in L^1(\mathbb{T}^5)\cap BMO,\\
        \sum_{k=0}^\infty\mathrm{curl}\,\phi_k^0\in L^1(\mathbb{T}^5)\cap BMO.
    \end{gather}
\end{prop}

\begin{proof}
    For simplicity, $\zeta_k$ denotes $\mathrm{curl}(\psi^0_k+\theta^0_k)$ or $\mathrm{curl}\,\phi_k^0$ 
    and $\zeta$ denotes $\sum_{k=0}^\infty\zeta_k$ in the following.
    Due to Proposition \ref{prop of Dpsi}, there exists a constant $C_\zeta$ such that $\|\zeta_k\|_{L^{\infty}(\mathbb{T}^5)}\leq C_\zeta$ for $k\in\mathbb{N}$. Recalling $k(Q)$ defined in Property 4 of Lemma \ref{lemma divfree}. Thus, using (\ref{e Omegak}) and the fact that $\{\zeta_k\}$ are smooth, we have
    \begin{align*}
        \sum_{k=0}^\infty \|\zeta_k\|_{L^{1}(\mathbb{T}^5)}&\leq \sum_{k=0}^{k(\mathbb{T}^5)-1}\|\zeta_k\|_{L^{1}(\mathbb{T}^5)}+\sum_{k=k(\mathbb{T}^5)}^\infty\|\zeta_k\|_{L^{\infty}(\mathbb{T}^5)}|\mathrm{supp}\,\zeta_k|\\
        &\lesssim \sum_{k=0}^{k(\mathbb{T}^5)-1} \|\zeta_k\|_{L^{1}(\mathbb{T}^5)}+\sum_{k=k(\mathbb{T}^5)}^\infty 2^{k(\mathbb{T}^5)-k}|\mathbb{T}^5|\\
        &<\infty.
    \end{align*}

  We recall the definition of \textit{BMO} space. For any $f\in L^1(\mathbb{T}^5)$, we identify $f$ as a periodic function on $\mathbb{R}^5$ and $f\in BMO$ if 
  \begin{gather*}
      \|f\|_{BMO}\triangleq\sup_{Q\subset\mathbb{R}^d}\fint_Q|f-f_Q|<\infty.
  \end{gather*}
  
  Fix a cube $Q\subset \mathbb{R}^d$. We proceed to estimate the quantity by dividing into two cases: 
  
  \noindent
  {\bf Case 1: $k(Q)=0$}: Suppose that $Q$ is decomposed into disjoint subsets as follows,
  \begin{gather*}
      Q=(\bigcup_{l=1}^{M}Q_l)\cup (\bigcup_{l=1}^{2^5(M+1)}\widetilde{Q}_l),
  \end{gather*}
  where $Q_l\pmod{(2\pi\mathbb{Z})^d}=\mathbb{T}^5$ and $\widetilde{Q}_l$ denotes the part that is not the 5-dimensional unit cube and interacts with other unit cubes. Here we take the number of sets $\widetilde{Q}_l$ as small as possible, so that it is less than $2^5(M+1)$.

Recalling $k(Q)= \inf\{k\in\mathbb{N}:|\Omega_k^3|\leq|Q|\}$, then we have $|Q|\geq|\widetilde{\Omega}_0|$. Thus, we obtain
\begin{align}
    \fint_Q|\zeta-\zeta_Q|&\leq\frac{2}{|Q|}\int_Q|\zeta|\nonumber\\&\leq\frac{2}{|Q|}\left(M\|\zeta\|_{L^1(\mathbb{T}^5)}+\sum_{l=1}^{2^5(M+1)}\int_{\widetilde{Q}_l}|\zeta|\right)\nonumber\\
    &\leq\frac{2}{|Q|}\left(2^5(M+1)+M\right)\|\zeta\|_{L^1(\mathbb{T}^5)}\nonumber\\
    &\leq
    \begin{cases}
        130|\mathbb{T}^5|^{-1}\|\zeta\|_{L^1(\mathbb{T}^5)}\quad \text{if } M\geq1,\\
        64|\widetilde{\Omega}_0|^{-1}\|\zeta\|_{L^1(\mathbb{T}^5)} \quad \text{if } M=0.
    \end{cases}\label{estimate case1}
\end{align}
Here we use $|Q|>M|\mathbb{T}^5|$ in the last inequality if $M\geq1$.

  \noindent
  {\bf Case 2: $k(Q)\geq1$}: In this case, we have $|Q|\leq|\widetilde{\Omega}_0|$, thus we can assume $Q\subset[0,2\pi]^5$ without loss of generality. Following the same idea of \cite[Proposition 3.8]{Coiculescu2025}, we estimate by dividing into two parts,
  \begin{gather*}
      \fint_Q|\zeta-\zeta_Q|\leq2\fint_Q\sum_{k\geq k(Q)-1}|\zeta_k|+\sum_{k\leq k(Q)-2}\|\zeta_k-(\zeta_k)_Q\|_{L^{\infty}(Q)}= I+II.
  \end{gather*}
  Note that $\|\zeta_k\|_{L^{\infty}(\mathbb{T}^5)}\leq C_\zeta$ for $k\in\mathbb{N}$ and (\ref{e Omegak}), we deduce
  \begin{gather}
      I\lesssim\sum_{k\leq k(Q)-1}\frac{|Q\cap\Omega_k^3|}{|Q|}\lesssim\sum_{k\leq k(Q)-1}2^{k(Q)-k}\lesssim1.\label{estimate I}
  \end{gather}
  To bound \textit{II}, we note that for any $k\leq k(Q)-2$,
  \begin{align*}
      |\zeta_k(x)-(\zeta_k)_Q|=\left|\fint_Q(\zeta_k(x)-\zeta(y))dy\right|\lesssim\|\nabla\zeta_k\|_{L^{\infty}(\mathbb{T}^5)}l(Q)\lesssim N_k |\widetilde{\Omega}_{k(Q)-1}|^{\frac{1}{5}}.
  \end{align*}
  Hence, we have
  \begin{gather}
      II\lesssim \sum_{k\leq k(Q)-2}N_k |\widetilde{\Omega}_{k(Q)-1}|^{\frac{1}{5}}\lesssim N_{k(Q)-2}(M_{k(Q)-1})^{\frac{3}{5}}\lesssim1.\label{estimate II}
  \end{gather}
  
  Combing (\ref{estimate case1}), (\ref{estimate I}) and (\ref{estimate II}) together, we can conclude that $\zeta\in BMO$.
\end{proof}

\subsection{Construction of the principal parts }
\begin{defi}\label{defi of v}
    Given $\psi^0_k$, $\phi^0_k$ and $\theta^0_k$ in Definition \ref{defi of psiok}, we define the initial data
    \begin{gather}
        U^0\triangleq\sum_{k=0}^{\infty}\mathrm{curl}\mathrm{curl}(\psi^0_k+\theta^0_k),\\
        B^0\triangleq \sum_{k=0}^{\infty}\mathrm{curl}\mathrm{curl}\,\phi^0_k.
    \end{gather}
    For any $k\geq0$, we define the heat-dominated evolution
    \begin{align}
        v_k&\triangleq\mathrm{currl}\mathrm{curl}(\psi_k+\theta_k)=\mathrm{currl}\mathrm{curl}(\psi^0_k+\theta^0_k)\mathrm{exp}(-N_k^2t),\\
        b_k&\triangleq\mathrm{currl}\mathrm{curl}\,\phi_k=\mathrm{currl}\mathrm{curl}\,\phi^0_k\mathrm{exp}(-N_k^2t),
    \end{align}
    and the inverse cascade-dominated evolution
    \begin{gather}
        \bar{v}_k\triangleq\frac{1}{2}N_{k+1}^{-2}\mathbb{P}\mathrm{div}\left(\sum_{\xi\in\Lambda_U}A_{\xi,k+1}a_{\xi,k+1}^2\xi_1\otimes\xi_1+\sum_{\xi\in\Lambda_B}B_{\xi,k+1}b_{\xi,k+1}^2(\xi_2\otimes\xi_2-\xi_1\otimes\xi_1)\right)\mathrm{exp}(-2N_{k+1}^2t),\\
        \bar{b}_k\triangleq\frac{1}{2}N_{k+1}^{-2}\mathrm{div}\sum_{\xi\in\Lambda_B}B_{\xi,k+1}b_{\xi,k+1}^2(\xi_2\otimes\xi_1-\xi_1\otimes\xi_2)\mathrm{exp}(-2N_{k+1}^2t).
    \end{gather}
    We define the principal parts of two solutions of (\ref{e:MHDe})
    \begin{gather}
    v^{(1)}\triangleq\sum_{k\geq0\,even}v_k+\sum_{k\geq0\,odd}\bar{v}_k,\quad b^{(1)}\triangleq\sum_{k\geq0\,even}b_k+\sum_{k\geq0\,odd}\bar{b}_k
    \end{gather}
    and 
    \begin{gather}
    v^{(2)}\triangleq\sum_{k\geq0\,odd}v_k+\sum_{k\geq0\,even}\bar{v}_k,\quad b^{(2)}\triangleq\sum_{k\geq0\,odd}b_k+\sum_{k\geq0\,even}\bar{b}_k.
    \end{gather}
\end{defi}

\begin{rema}\label{intialequal}
    Let us show that $(v^{(1)},b^{(1)})(0,\cdot)=(v^{(2)},b^{(2)})(0,\cdot)$.
    To prove this, it suffices to verify that $\bar{v}_k(0,\cdot)=v_k(0,\cdot)$ and $\bar{b}_k(0,\cdot)=b_k(0,\cdot)$.
    By the definition of $a_{\xi,k+1}$ and $s_{k+1}$, we compute
    \begin{align*}
        \bar{v}_k(0,\cdot)&=\mathbb{P}\mathrm{div}\left(\sum_{\xi\in\Lambda_U}l_{k+1}\chi_{k+1}^2\Gamma_\xi^2(\mathrm{Id}+l_{k+1}^{-1}(\mathcal{D}\psi^0_{k}+\mathcal{D}\theta^0_{k}-s_{k+1}))\xi_1\otimes\xi_1+s_{k+1}\right)\\
        &=\mathbb{P}\mathrm{div}\left(\chi_{k+1}^V2(\mathcal{D}\psi^0_{k}+\mathcal{D}\theta^0_{k}-s_{k+1})+s_{k+1}\right)\\
        &=\mathbb{P}\mathrm{div}(\mathcal{D}\psi^0_{k}+\mathcal{D}\theta^0_{k})\\
        &=v_k(0,\cdot)
    \end{align*}
    where we use the fact that $\chi_{k+1}^V$ is identically equal to 1 on support of $\psi^0_{k}$, $\theta^0_{k}$ and $s_{k+1}$ and Lemma \ref{lemmaNS}. Similarly, it is straightforward to check  $\bar{b}_k(0,\cdot)=b_k(0,\cdot)$ by applying the properties of $\chi_{k+1}^B$ and Lemma \ref{lemmaMHD}. We omit the details here.
\end{rema}

\begin{lemm}\label{lemma vk}
    Let $\beta>0$. For any $k,m\in\mathbb{N}$, $v_k$, $\bar{v}_k$, $b_k$ and $\bar{b}_k$ are divergence-free and satisfy the following estimates
    \begin{gather}
        \|\nabla^m(v_k,b_k)\|_{_{L^{\infty}(\mathbb{T}^5)}}\lesssim N_k^{1+m}\mathrm{exp}({-N^2_kt}),\label{e vk}\\
        \|\nabla^m(\bar{v}_k,\bar{b}_k)\|_{_{L^{\infty}(\mathbb{T}^5)}}\lesssim N_k^{1+\beta+m}\mathrm{exp}({-N^2_{k+1}t}).\label{e barvk}
    \end{gather}
\end{lemm}

\begin{proof}
    The above estimates are directly derived from Proposition \ref{prop nablampsi} and Definition \ref{defi of v}. More precisely, since $\mathbb{P}$ is not bounded on $L^\infty(\mathbb{T}^5)$, we estimate $\bar{v}_k$ on $C^{\beta}(\mathbb{T}^5)$ instead of $L^\infty(\mathbb{T}^5)$. Thus, $\beta$ appears in (\ref{e barvk}). 
\end{proof}

Applying Lemma \ref{lemma vk} with $\beta<\frac{b-1}{4}\wedge\frac{1}{2}$, we can obtain the following results. Since the proof closely parallels that of \cite[Proposition 3.13]{Coiculescu2025}, the detailed verification is omitted for brevity.

\begin{prop}\label{prop of v}
    For any $i\in\{1,2\}, $ $(v^{(i)},b^{(i)})$ are divergence-free and obey the following estimates
    \begin{gather}
        \|(\nabla^mv^{(i)},\nabla^mb^{(i)})\|_{L^{\infty}(\mathbb{T}^5)}\lesssim_m t^{-\frac{m+1}{2}}\mathrm{exp}({-C(m)N^2_0t})\label{1 estimate1 of v}
    \end{gather}
    for all $m\geq0$ and $t>0$. Moreover, we have
    \begin{gather}
        \|(v^{(i)},b^{(i)})\|_{L^2([t_1,t_2],dt;L^{\infty})}^2+\|(v^{(i)},b^{(i)})\|_{L^1([t_1,t_2],t^{-\frac{1}{2}}dt;L^{\infty})}\lesssim1+(\mathrm{logA})^{-1}(\frac{t_2}{t_1})\label{2 estimate of v}
    \end{gather}
    for any $0<t_1\leq t_2\leq1$.
\end{prop}

\section{Construction of the perturbation}\label{sec:4}
 
 $(v^{(1)},b^{(1)})$ and $(v^{(2)},b^{(2)})$ we constructed in the above section are almost solutions of (\ref{e:MHDe}). To show this, we prove that the residual terms are small in a subcritical norm. Consequently, for $i\in\{1,2\}$, we can add a perturbation to $(v^{(i)}, b^{(i)})$ to obtain an exact solution.
\subsection{Estimates on the residual}\label{sec:4.1}
First of all, let us introduce a subcritical norm to measure the size of the residual terms.
Recalling $\alpha \in (0,\frac{1}{8})$ and $\kappa \in (0,\frac{1}{2}-\frac{1}{2\gamma}-2\alpha)$, we consider the Banach space
\begin{align*}
    Y=\{(a,b) \in C^0((0,1];C^{1,\kappa}(\mathbb{T}^5;S^{5 \times 5})) \times C^0((0,1];C^{1,\kappa}(\mathbb{T}^5;A^{5 \times 5}))\},
\end{align*}
with norm
\begin{align*}
    \| (a,b)\|_Y:=\sup_{t \in (0,1]}(t^{1-\alpha}(\| a\|_{L^{\infty}}+\| b\|_{L^{\infty}})+t^{\frac{3}{2}-\alpha}(\| \nabla a\|_{C^k(T^5)}+\| \nabla b\|_{C^k(T^5)})) < \infty.
\end{align*}

For $i\in\{1,2\}$, we define the residual terms $(F^{(i)},G^{(i)})$ in divergence form:
\begin{align*}
    \begin{cases}
        -\mathbb{P}\mathrm{div}F^{(i)}=(\partial_t-\Delta) v^{(i)}+ \mathbb{P}\text{div}( v^{(i)}\otimes  v^{(i)}-b^{(i)}\otimes b^{(i)}),  \\
        -\mathrm{div}G^{(i)}=(\partial_t-\Delta) b^{(i)}+\mathrm{div}(v^{(i)}\otimes  b^{(i)}-b^{(i)}\otimes v^{(i)}).
    \end{cases}
\end{align*}

For the purpose of estimating $(F^{(i)},G^{(i)})$, we decompose $(v_k,b_k)=(v^p_k+v^e_k,b^p_k+b^e_k)$. The leading part $(v^p_k,b^p_k)$ is defined as follows:
\begin{align*}
    v^p_k&\triangleq N_k^2(\psi^0_k+\theta^0_k)\mathrm{exp}(-N_k^2t),\\
    b^p_k&\triangleq N_k^2\phi^0_k\mathrm{exp}(-N_k^2t).
\end{align*}
And the minor part $(v^e_k,b^e_k)$ is given by
\begin{align*}
    v^e_k&\triangleq v^{e,1}_k+v^{e,2}_k+v^{e,3}_k+v^{e,4}_k,\\
    b^e_k&\triangleq b^{e,1}_k+b^{e,2}_k+b^{e,3}_k+b^{e,4}_k,
\end{align*}
where
\begin{align*}
    v^{e,1}_k&\triangleq(\eta_k-\delta)*\mathrm{curl}\mathrm{curl}\left(\sum_{\xi\in\Lambda_U}a_{\xi,k}\Psi^0_{\xi,k}+\sum_{\xi\in\Lambda_B}b_{\xi,k}\Theta^0_{\xi,k}\right)\mathrm{exp}(-N_k^2t),\\
    v^{e,2}_k&\triangleq-N_k^{-2}\mathrm{exp}(-N_k^2t)\times\\
    &\quad\left(\sum_{\xi\in\Lambda_U}\Delta(a_{\xi,k}\varphi_{\xi,k})\mathrm{sin}(N_k(x-x_\xi)\cdot \xi)\xi_1+\sum_{\xi\in\Lambda_B}\Delta(b_{\xi,k}\varphi_{\xi,k})\mathrm{sin}(N_k(x-x_\xi)\cdot \xi)\xi_2\right),\\
    v^{e,3}_k&\triangleq -2N_k^{-1}\mathrm{exp}(-N_k^2t)\times
    \\&\quad\left(\sum_{\xi\in\Lambda_U}\xi\cdot\nabla(a_{\xi,k}\varphi_{\xi,k})\mathrm{cos}(N_k(x-x_\xi)\cdot \xi)\xi_1+\sum_{\xi\in\Lambda_B}\xi\cdot\nabla(b_{\xi,k}\varphi_{\xi,k})\mathrm{cos}(N_k(x-x_\xi)\cdot \xi)\xi_2\right),\\
    v^{e,4}_k&\triangleq N_k^{-2}\mathrm{exp}(-N_k^2t)\times
    \\&\quad\nabla\left(\sum_{\xi\in\Lambda_U}\xi_1\cdot\nabla a_{\xi,k}\varphi_{\xi,k}\mathrm{cos}(N_k(x-x_\xi)\cdot \xi)+\sum_{\xi\in\Lambda_B}\xi_2\cdot\nabla b_{\xi,k}\varphi_{\xi,k}\mathrm{cos}(N_k(x-x_\xi)\cdot \xi)\right),
\end{align*}
and
\begin{flalign*}
    &b^{e,1}_k\triangleq(\eta_k-\delta)*\mathrm{curl}\mathrm{curl}\sum_{\xi\in\Lambda_B}b_{\xi,k}\Phi^0_{\xi,k}\mathrm{exp}(-N_k^2t),&\\
   &b^{e,2}_k\triangleq-N_k^{-2}\sum_{\xi\in\Lambda_B}\Delta(b_{\xi,k}\varphi_{\xi,k})\mathrm{sin}(N_k(x-x_\xi)\cdot \xi)\xi_1\mathrm{exp}(-N_k^2t),&\\
    &b^{e,3}_k\triangleq-2N_k^{-1}\sum_{\xi\in\Lambda_B}\xi\cdot\nabla(b_{\xi,k}\varphi_{\xi,k})\mathrm{cos}(N_k(x-x_\xi)\cdot \xi)\xi_1\mathrm{exp}(-N_k^2t),&\\
    &b^{e,4}_k\triangleq N_k^{-2}\nabla\sum_{\xi\in\Lambda_B}\xi_1\cdot\nabla b_{\xi,k}\varphi_{\xi,k}\mathrm{cos}(N_k(x-x_\xi)\cdot \xi)\mathrm{exp}(-N_k^2t).
\end{flalign*}

The estimates for $(v^p_k,b^p_k)$ is of the same order as those for $(v_k,b_k)$ while the estimates of $(v^e_k,b^e_k)$ is of lower order than those for $(v^p_k,b^p_k)$. 
\begin{lemm}\label{lemmvp}
    For any $k,m\in\mathbb{N}$, we have
    \begin{align}
        \|\nabla^m(v^p_k,b^p_k)\|_{L^{\infty}(\mathbb{T}^5)}\lesssim N_{k}^{1+m}\mathrm{exp}(-N_k^2t),\label{e vpk}\\
        \|\nabla^m(v^e_k,b^e_k)\|_{L^{\infty}(\mathbb{T}^5)}\lesssim M_kN_{k}^{m}\mathrm{exp}(-N_k^2t).\label{e vek}
    \end{align}
\end{lemm}

\begin{proof}
    (\ref{e nablampsi}) and the definition of $(v^p_k,b^p_k)$ imply (\ref{e vpk}) immediately. In the following, we focus on the proof of the estimates on $b^e_k$. The proof of the estimates on $v^e_k$ is similar and we omit it.
    
    Applying (\ref{e varphi}), (\ref{e Phi}), (\ref{e nablbma}), the standard mollifier estimate, and the Leibnitz rule, we obtain
    \begin{align*}
        \|\nabla^m b^{e,1}_k\|_{L^{\infty}(\mathbb{T}^5)}&\lesssim N_{k+1}^{-\frac{1}{3}}N_k^{-\frac{2}{3}}\|\nabla^{m+3}\sum_{\xi\in\Lambda_B}b_{\xi,k}\Phi^0_{\xi,k}\|_{L^{\infty}(\mathbb{T}^5)}\mathrm{exp}(-N_k^2t)\\
        &\lesssim N_{k+1}^{-\frac{1}{3}}N_k^{\frac{4}{3}+m}\mathrm{exp}(-N_k^2t),
     \end{align*}
     and 
     \begin{align*}
         \|\nabla^m b^{e,2}_k\|_{L^{\infty}(\mathbb{T}^5)}&\lesssim M_k^2N_k^{-1+m}\mathrm{exp}(-N_k^2t),\\
         \|\nabla^m b^{e,3}_k\|_{L^{\infty}(\mathbb{T}^5)}&\lesssim M_kN_k^{m}\mathrm{exp}(-N_k^2t),\\
         \|\nabla^m b^{e,4}_k\|_{L^{\infty}(\mathbb{T}^5)}&\lesssim N_{k-1}N_k^{m}\mathrm{exp}(-N_k^2t).
     \end{align*}
     Combining the above estimates, we conclude that
     \begin{align*}
         \|\nabla^m b^{e}_k\|_{L^{\infty}(\mathbb{T}^5)}&\lesssim\|\nabla^m b^{e,1}_k\|_{L^{\infty}(\mathbb{T}^5)}+\|\nabla^m b^{e,2}_k\|_{L^{\infty}(\mathbb{T}^5)}+\|\nabla^m b^{e,3}_k\|_{L^{\infty}(\mathbb{T}^5)}+\|\nabla^m b^{e,4}_k\|_{L^{\infty}(\mathbb{T}^5)}\\
         &\lesssim M_kN_{k}^{m}\mathrm{exp}(-N_k^2t).
     \end{align*}
\end{proof}

In the following, we consider the nonlinear part of $(v^p_k,b^p_k)$. Using the fact that $\mathrm{supp}\varphi_{\xi,k}\cap\mathrm{supp}\varphi_{\xi',k}=\emptyset$ for any distinct $\xi,\xi'\in\Lambda_U\cup\Lambda_B$, we proceed to decompose the nonlinear term as follows:
\begin{align*}
    v^p_k\otimes v^p_k-b^p_k\otimes b^p_k&=N_k^4\left(\sum_{\xi\in\Lambda_U}(a_{\xi,k}\Psi^0_{\xi,k})^2+\sum_{\xi\in\Lambda_B}(b_{\xi,k}\Theta^0_{\xi,k})^2-\sum_{\xi\in\Lambda_B}(b_{\xi,k}\Phi^0_{\xi,k})^2\right)\mathrm{exp}(-2N_k^2t)\\
    &=\mathcal{N}_{k,1}+\mathcal{N}_{k,2},\\
    v^p_k\otimes b^p_k-b^p_k\otimes v^p_k&=N_k^4\sum_{\xi\in\Lambda_B}b_{\xi,k}^2\left(\Theta^0_{\xi,k}\otimes\Phi^0_{\xi,k}-\Phi^0_{\xi,k}\otimes\Theta^0_{\xi,k}\right)\mathrm{exp}(-2N_k^2t)\\
    &=\mathcal{M}_{k,1}+\mathcal{M}_{k,2},
\end{align*}
where
\begin{align*}
    \mathcal{N}_{k,1}&\triangleq\left(\sum_{\xi\in\Lambda_U}a_{\xi,k}^2A_{\xi,k}\xi_1\otimes\xi_1+\sum_{\xi\in\Lambda_B}b_{\xi,k}^2B_{\xi,k}\left(\xi_2\otimes\xi_2-\xi_1\otimes\xi_1\right)\right)\mathrm{exp}(-2N_k^2t),\\
    \mathcal{N}_{k,2}&\triangleq\sum_{\xi\in\Lambda_U}a_{\xi,k}^2\left(\varphi_{\xi,k}^2(x)\mathrm{sin}^2(N_k(x-x_\xi)\cdot \xi)-A_{\xi,k}\right)\xi_1\otimes\xi_1\mathrm{exp}(-2N_k^2t)\\
 &\quad+\sum_{\xi\in\Lambda_B}b_{\xi,k}^2\left(\varphi_{\xi,k}^2(x)\mathrm{sin}^2(N_k(x-x_\xi)\cdot \xi)-B_{\xi,k}\right)\left(\xi_2\otimes\xi_2-\xi_1\otimes\xi_1\right)\mathrm{exp}(-2N_k^2t),
\end{align*}
and
\begin{align*}
    \mathcal{M}_{k,1}&\triangleq\sum_{\xi\in\Lambda_B}b_{\xi,k}^2B_{\xi,k}\left(\xi_2\otimes\xi_1-\xi_1\otimes\xi_2\right)\mathrm{exp}(-2N_k^2t),\\
    \mathcal{M}_{k,2}&\triangleq\sum_{\xi\in\Lambda_B}b_{\xi,k}^2\left(\varphi_{\xi,k}^2(x)\mathrm{sin}^2(N_k(x-x_\xi)\cdot \xi)-B_{\xi,k}\right)\left(\xi_2\otimes\xi_1-\xi_1\otimes\xi_2\right)\mathrm{exp}(-2N_k^2t).
\end{align*}
Let us emphasize that $\mathcal{N}_{k,1}$ and $\mathcal{M}_{k,1}$ are principle parts of the nonlinear term, which will be cancelled by the inverse cascade-dominated flows $(\bar{v}_k,\bar{b}_k)$.

We recall two types of antidivergence operators $\mathcal{R}^V$ and $\mathcal{R}^B$ in Proposition \ref{def of antidiv}, and we define
\begin{align*}
    F^{(1)}_h\triangleq-\sum_{k\geq0,even}\mathcal{R}^V\mathrm{div}\mathcal{N}_{k,2},
    \quad G^{(1)}_h\triangleq-\sum_{k\geq0,even}\mathcal{R}^B\mathrm{div}\mathcal{M}_{k,2}.
\end{align*}
and
\begin{align*}
    F^{(2)}_h=-\sum_{k\geq0,odd}\mathcal{R}^V\mathrm{div}\mathcal{N}_{k,2},
    \quad G^{(2)}_h=-\sum_{k\geq0,odd}\mathcal{R}^B\mathrm{div}\mathcal{M}_{k,2}.
\end{align*}

Thus, we have the following estimates.
\begin{lemm}\label{Fh}
    For any $\varepsilon_0>0$ and $i\in\{1,2\}$, let $A$ be sufficiently large, then we have
    \begin{align}
        \|( F^{(i)}_h, G^{(i)}_h)\|_Y\leq\varepsilon_0.
    \end{align}
\end{lemm}
\begin{proof}
    As $\mathcal{R}^V$ and $\mathcal{R}^B$ both act as an operator of degree 1, with $\mathcal{N}_{k,2}$ and $\mathcal{M}_{k,2}$ being of similar form, we only need to prove the estimate on  $G^{(1)}_h$ and the other follows similarly.
    For any $\xi\in\Lambda_B$ and $k\geq0$, we denote 
    \begin{align*}
        h_{\xi,k}(x)=\varphi_{\xi}^2(x)\mathrm{sin}^2(\frac{N_k}{M_k}x\cdot  \xi)-B_{\xi,k},
    \end{align*}
    where $\varphi_{\xi}$ is defined in (\ref{defivarphixi}).
    
    Note that $h_{\xi,k}$ is mean-free, we take the fournier form of $h_{\xi,k}$:
    \begin{gather*}
        h_{\xi,k}(x)=\sum_{q\in\mathbb{Z}^d\setminus \{0\}}\hat{h}_{\xi,k}(q)e^{-i qx}
    \end{gather*}
    Thus, by Property 2 in Lemma \ref{lemma divfree}, we can rewrite $G^{(1)}_h$:
    \begin{gather}
        G^{(1)}_h=-\sum_{k\geq0,even}\sum_{\xi\in\Lambda_B}\sum_{q\in\mathbb{Z}^d\setminus \{0\}}\hat{h}_{\xi,k}(q)\mathcal{R}^B\left(\mathrm{div}(b_{\xi,k}^2\left(\xi_2\otimes\xi_1-\xi_1\otimes\xi_2\right))e^{-i qM_k(x-x_\xi)}\right)\mathrm{exp}(-2N_k^2t).\label{canti0}
    \end{gather}
    The function $x\colon\mapsto\mathrm{sin}^2(\frac{N_k}{M_k}x\cdot  \xi)$ has Fouriner support at just $q=0$ and $q=\pm2\frac{N_k}{M_k}\xi$, this implies
    \begin{gather*}
        \hat{h}_{\xi,k}(q)=\frac{1}{2}\mathcal{F}(\varphi^2_\xi)(q)-\frac{1}{4}\mathcal{F}(\varphi^2_\xi)(q+2\frac{N_k}{M_k}\xi)-\frac{1}{4}\mathcal{F}(\varphi^2_\xi)(q-2\frac{N_k}{M_k}\xi),\quad\forall q\in\mathbb{Z}^d\setminus \{0\}.
    \end{gather*}
    Recall that $\varphi^2_\xi$ is smooth. there holds
    \begin{gather}
        |\hat{h}_{\xi,k}(q)|\lesssim\langle q\rangle^{-10}+\langle q+2\frac{N_k}{M_k}\xi\rangle^{-10}+\langle q-2\frac{N_k}{M_k}\xi\rangle^{-10}.\label{canti1}
    \end{gather}
    Appying (\ref{e nablbma}) and (\ref{RBestimate}) with $\beta=\frac{1}{2}-\frac{\gamma}{2}(4\alpha+b^{-1})$, we obtain
    \begin{align}
        &\quad\quad\|\mathcal{R}^B\left(\mathrm{div}\left(b_{\xi,k}^2(\xi_2\otimes\xi_1-\xi_1\otimes\xi_2)\right)e^{-i qM_k(x-x_\xi)}\right)\|_{C^\beta(\mathbb{T}^5)}\nonumber\\
        &\lesssim_mN_{k-1}N_k^2(M_k^{-1+\beta}+M_k^{-m+\beta}N_{k-1}^{m}+M_k^{-m}N_{k-1}^{m+\beta})\nonumber\\&\lesssim_m N_{k-1}N_k^2M_k^{-1+\beta},\label{canti2}
    \end{align}
    where we choose $m>b/(b-\gamma)$ so that $(\frac{N_{k-1}}{M_k})^m\leq M_k^{-1}$.
    Combining (\ref{canti0}), (\ref{canti1}) and (\ref{canti2}), it holds
    \begin{align*}
        \|G^{(1)}_h(t)\|_{L^\infty(\mathbb{T}^5)}&\lesssim\sum_{k\geq0,even}N_{k-1}N_k^2M_k^{-1+\beta}\mathrm{exp}(-2N_k^2t)\nonumber\\&\quad\quad\times\sum_{q\in\mathbb{Z}^d\setminus \{0\}}\left(\langle q\rangle^{-10}+\langle q+2\frac{N_k}{M_k}\xi\rangle^{-10}+\langle q-2\frac{N_k}{M_k}\xi\rangle^{-10}\right)\nonumber\\
        &\lesssim\sup_k(N_{k-1}M_k^{-1+\beta})^{\frac{1}{2}}\sum_{k\geq0,even}(N_{k-1}M_k^{-1+\beta})^{\frac{1}{2}}N_k^2\mathrm{exp}(-2N_k^2t).
    \end{align*}
    By definition of $\beta$, we obtain
    \begin{gather*}
        (N_{k-1}M_k^{-1+\beta})^{\frac{1}{2}}\leq N_k^{\frac{1}{2b}+\frac{-1+\beta}{2\gamma}}\leq N_k^{-2\alpha},
    \end{gather*}
    which immediately implies
    \begin{gather}
        \|G^{(1)}_h(t)\|_{L^\infty(\mathbb{T}^5)}\leq \varepsilon_0t^{-1+\alpha}\label{nablag1h1}
    \end{gather}
    if $A$ is large enough.
    
    Using the fact that $\nabla\mathcal{R}^B$ is bounded on $C^\kappa(\mathbb{T}^5)$ and (\ref{e nablbma}), we have
    \begin{align}
        \|\nabla G^{(1)}_h\|_{C^\kappa(\mathbb{T}^5)}&\lesssim\sum_{k\geq0,even}\sum_{\xi\in\Lambda_B}\|\nabla(b_{\xi,k}^2)(\varphi_{\xi,k}^2(x)\mathrm{sin}^2(N_k(x-x_\xi)\cdot  \xi)-B_{\xi,k})\|_{C^\kappa(\mathbb{T}^5)}\mathrm{exp}(-2N_k^2t)\nonumber\\
        &\lesssim \sum_{k\geq0,even}N_{k-1}N_k^{2+\kappa}\mathrm{exp}(-2N_k^2t)\nonumber\\
        &\lesssim \sup_k(N_{k-1}N_k^{-1+\kappa})^\frac{1}{2}\sum_{k\geq0,even}(N_{k-1}N_k^{-1+\kappa})^\frac{1}{2}N_k^{3}\mathrm{exp}(-2N_k^2t).
    \end{align}
    By definition of $\kappa$, we have
    \begin{gather*}
        (N_{k-1}N_k^{-1+\kappa})^\frac{1}{2}\leq N_k^{\frac{1}{2}(b^{-1}-1+\kappa)}\leq N_k^{-2\alpha}.
    \end{gather*}
    Thus, we obtain
    \begin{gather}
        \|\nabla G^{(1)}_h(t)\|_{C^\kappa(\mathbb{T}^5)}\leq \varepsilon_0t^{-\frac{3}{2}+\alpha}\label{nablag1h2}
    \end{gather}
    provided that $A$ is sufficiently large.
    
    Combing (\ref{nablag1h1}) and (\ref{nablag1h2}) together and taking $A$ sufficiently large, we conclude that
    \begin{gather*}
        t^{1-\alpha}\|G^{(1)}_h\|_{L^\infty(\mathbb{T}^5)}+t^{\frac{3}{2}-\alpha}\|\nabla G^{(1)}_h\|_{C^\kappa(\mathbb{T}^5)}\leq\varepsilon_0.
    \end{gather*}
\end{proof}

Recall Definition \ref{defi of psiok}, let us define the tensor fields
\begin{align*}
    S_k&\triangleq\mathcal{D}(\psi_k^0+\theta_k^0)\mathrm{exp}(-N_k^2t),\\
    \bar{S}_k&\triangleq\frac{1}{2}N_{k+1}^{-2}\left(\sum_{\xi\in\Lambda_U}A_{\xi,k+1}a_{\xi,k+1}^2\xi_1\otimes\xi_1+\sum_{\xi\in\Lambda_B}B_{\xi,k+1}b_{\xi,k+1}^2(\xi_2\otimes\xi_2-\xi_1\otimes\xi_1)\right)\mathrm{exp}(-2N_{k+1}^2t),\\
    T_k&\triangleq\widetilde{\mathcal{D}}\phi_k^0\mathrm{exp}(-N_k^2t),\\
    \bar{T}_k&\triangleq\frac{1}{2}N_{k+1}^{-2}\sum_{\xi\in\Lambda_B}B_{\xi,k+1}b_{\xi,k+1}^2(\xi_2\otimes\xi_1-\xi_1\otimes\xi_2)\mathrm{exp}(-2N_{k+1}^2t).
\end{align*}
Due to Definition \ref{defi of v}, it is easy to see
\begin{align*}
\mathbb{P}\mathrm{div}S_k=v_k,\quad\mathbb{P}\mathrm{div}\bar{S}_k=\bar{v}_k,\quad\mathrm{div}\,T_k=b_k\quad\text{and}\quad\mathrm{div}\,\bar{T}_k=\bar{b}_k.
\end{align*}
Moreover, by the same computation of Remark \ref{intialequal}, we have 
\begin{align}
    \partial_t\bar{S}_k=-\mathcal{N}_{k+1,1}\quad\text{and}\quad\partial_t\bar{T}_k=-\mathcal{M}_{k+1,1}.\label{dixiao}
\end{align}

Finally, we decompose $(F^{(1)},G^{(1)})$ as follows:
\begin{align*}
    F^{(1)}= F^{(1)}_1+F^{(1)}_2+F^{(1)}_3+F^{(1)}_4+F^{(1)}_h,\\
    G^{(1)}= G^{(1)}_1+G^{(1)}_2+G^{(1)}_3+G^{(1)}_4+G^{(1)}_h,\\
\end{align*}
where
\begin{align*}
    F^{(1)}_1&\triangleq-\sum_{k\geq0,even}(\partial_t-\Delta)S_k,\\
    F^{(1)}_2&\triangleq\sum_{k\geq0,odd}\Delta\bar{S}_k,\\
    F^{(1)}_3&\triangleq-\sum_{k\geq0,odd}(\partial_t\bar{S}_k+\mathcal{N}_{k+1,1})-\mathcal{R}^V\mathrm{div}\mathcal{N}_{0,1},\\
    F^{(1)}_4&\triangleq-v^{(1)}\otimes v^{(1)}+b^{(1)}\otimes b^{(1)}+\sum_{k\geq0,even}(v^{p}_k\otimes v^{p}_{k}-b^{p}_k\otimes b^{p}_{k}),
\end{align*}
and
\begin{align*}
    G^{(1)}_1&\triangleq-\sum_{k\geq0,even}(\partial_t-\Delta)T_k,\\
    G^{(1)}_2&\triangleq\sum_{k\geq0,odd}\Delta\bar{T}_k,\\
    G^{(1)}_3&\triangleq-\sum_{k\geq0,odd}(\partial_t\bar{T}_k+\mathcal{M}_{k+1,1})-\mathcal{R}^B\mathrm{div}\mathcal{M}_{0,1},\\
    G^{(1)}_4&\triangleq-v^{(1)}\otimes b^{(1)}+b^{(1)}\otimes v^{(1)}+\sum_{k\geq0,even}(v^{p}_k\otimes b^{p}_{k}-b^{p}_k\otimes v^{p}_{k}).
\end{align*}
Based on the definition of $(v^{(2)},b^{(2)})$, we can construct $(F^{(2)},G^{(2)})$ analogously. Here we omit the details. In the following, we prove $\{(F^{(i)},G^{(i)})\}_{i\in\{\{1,2\}}$ are arbitrary small in the subcritical space $Y$.

\begin{prop}\label{Prop 4.1}
    For any $\varepsilon_0>0$ and $i\in\{1,2\}$, we have
    \begin{gather}
        \|(F^{(i)},G^{(i)})\|_{Y}\leq \varepsilon_0,
    \end{gather}
    if $A$ is sufficiently large.
\end{prop}

\begin{proof}
    For $i\in\{1,2\}$, we decompose the pair $(F^{(i)},G^{(i)})$ into structures of a similar form, thus we only prove the esimates on $F^{(1)}$.

   Firstly, we establish the estimate for $F^{(1)}_1$. Observing that $\mathrm{sin}(N_k(x-x_\xi)\cdot  \xi)\mathrm{exp}(-N_{k}^2t)$ is a solution of the heat equation, we obtain
    \begin{align*}
        &(\partial_t-\Delta)S_k\\
        =&N_k^{-2}\eta_k*\mathcal{D}\left(\sum_{\xi\in\Lambda_U}\Delta(a_{\xi,k}\varphi_{\xi,k})\mathrm{sin}(N_k(x-x_\xi)\cdot  \xi)\xi_1+2N_k\xi\cdot\nabla(a_{\xi,k}\varphi_{\xi,k})\mathrm{cos}(N_k(x-x_\xi)\cdot  \xi)\xi_1\right.\\
      &+\left.\sum_{\xi\in\Lambda_B}\Delta(b_{\xi,k}\varphi_{\xi,k})\mathrm{sin}(N_k(x-x_\xi)\cdot  \xi)\xi_2+2N_k\xi\cdot\nabla(b_{\xi,k}\varphi_{\xi,k})\mathrm{cos}(N_k(x-x_\xi)\cdot  \xi)\xi_2\right)\mathrm{exp}(-N_{k}^2t).
     \end{align*}
     Thus, using (\ref{e varphi}) and (\ref{e nablbma}), we deduce
     \begin{align*}
         \|\nabla^mF^{(1)}_1\|_{L^\infty(\mathbb{T}^5)}&\lesssim \sum_{k\geq0,even} M_kN_{k}^{1+m}\mathrm{exp}(-N_{k}^2t). 
     \end{align*}
     Hence, we have
     \begin{align}
         t\|F^{(1)}_1\|_{L^\infty(\mathbb{T}^5)}+t^{\frac{3}{2}}\|\nabla F^{(1)}_1\|_{L^\infty(\mathbb{T}^5)}&\lesssim \sup_k(\frac{M_kN_k^\kappa}{N_k})^{\frac{1}{2}}t^{-\frac{1}{2}(\frac{1}{\gamma}+\kappa-1)}\leq\frac{\varepsilon_0}{5}t^\alpha\label{F1}
     \end{align}
     if $A$ is chosen sufficiently large.
    
    Next, consider $F^{(1)}_2$. Once again, by (\ref{e nablbma}) and the definition of $F^{(1)}_2$, we get
    \begin{align*}
        \|\nabla^mF^{(1)}_2\|_{L^\infty(\mathbb{T}^5)}\lesssim \sum_{k\geq0,odd}N_{k}^{2+m}\mathrm{exp}(-N_{k+1}^2t).
    \end{align*}
    Similarly with the estimes of $F^{(1)}_1$, we obtain
    \begin{align}
         t\|F^{(1)}_2\|_{L^\infty(\mathbb{T}^5)}+t^{\frac{3}{2}}\|\nabla F^{(1)}_2\|_{L^\infty(\mathbb{T}^5)}\leq\frac{\varepsilon_0}{5}t^\alpha\label{F2}
     \end{align}
     provided that $A$ is chosen sufficiently large.

    For $F^{(1)}_3$, we claim that $F^{(1)}_3$ is actually 0. In fact, due to Definition \ref{definition of phi}, we know $\{a_{\xi,0}\}_{\Lambda_U}$ and $\{b_{\xi,0}\}_{\Lambda_B}$ are $N_0$ or 0. Thus, we have $\mathrm{div}\mathcal{N}_{0,1}=0$. Combing (\ref{dixiao}) together, we know $F^{(1)}_3=0$.

    Finally, we focus on $F^{(1)}_4$. We decompose $v^{(1)}$ and $b^{(1)}$ into two parts:
    \begin{align*}
        v^{(1)}=v^{p}+\widetilde{v}\quad\mathrm{and}\quad b^{(1)}=b^{p}+\widetilde{b},
    \end{align*}
    where 
    \begin{gather*}
        v^{p}=\sum_{k\geq0,even}v^p_k,\quad b^{p}=\sum_{k\geq0,even}b^p_k,\\
        \widetilde{v}=\sum_{k\geq0,even}v^e_k+\sum_{odd}\bar{v}_k,\quad\widetilde{b}=\sum_{k\geq0,even}b^e_k+\sum_{odd}\bar{b}_k.
    \end{gather*}
    Thus, we can spilt
    \begin{align*}
        F^{(1)}_4&=\left(-v^{p}\otimes\widetilde{v}-\widetilde{v}\otimes v^{(1)}+b^{p}\otimes\widetilde{b}+\widetilde{b}\otimes b^{(1)}\right)+\left(-\sum_{k\neq j,even}v^{p}_k\otimes v^{p}_j+\sum_{k\neq j,even}b^{p}_k\otimes b^{p}_j\right)\\
        &\triangleq F^{(1)}_{4,1}+F^{(1)}_{4,2}.
    \end{align*}
    From (\ref{e vk}), (\ref{e barvk}), (\ref{e vpk}) and (\ref{e vek}), it yields that
    \begin{gather*}
        \|\nabla^mF^{(1)}_{4,1}\|_{L^\infty(\mathbb{T}^5)}\lesssim\sum_{k}N_k^{1+m}\mathrm{exp}(-N_k^2t)\times\left(\sum_{k}N_k^{2}\mathrm{exp}(-N_{k+1}^2t)+\sum_{k}M_k\mathrm{exp}(-N_k^2t)\right),
    \end{gather*}
    Here we apply (\ref{e barvk}) with choosing $\beta=1$. By the same reasoning as before, this implies
    \begin{align}
         t\|F^{(1)}_{4,1}\|_{L^\infty(\mathbb{T}^5)}+t^{\frac{3}{2}}\|\nabla F^{(1)}_{4,1}\|_{L^\infty(\mathbb{T}^5)}\leq\frac{\varepsilon_0}{5}t^\alpha.\label{F41}
     \end{align}
    
    Consider $F^{(1)}_{4,2}$, we deduce 
    \begin{align*}
        \|\nabla^mF^{(1)}_{4,2}\|_{L^\infty(\mathbb{T}^5)}&\lesssim \sum_{k\neq j,even}(N_k+N_j)^mN_kN_j\mathrm{exp}(-(N_{k}^2+N_{j}^2)t)\nonumber\\
        &\lesssim\sum_{k}N_k^{1+m}\mathrm{exp}(-N_k^2t)\sum_{j<k}N_j\nonumber\\
        &\lesssim\sum_{k}N_k^{1+m}N_{k-1}\mathrm{exp}(-N_k^2t),
    \end{align*}
    which gives 
    \begin{align}
         t\|F^{(1)}_{4,2}\|_{L^\infty(\mathbb{T}^5)}+t^{\frac{3}{2}}\|\nabla F^{(1)}_{4,2}\|_{L^\infty(\mathbb{T}^5)}\leq\frac{\varepsilon_0}{5}t^\alpha.\label{F42}
     \end{align}
    
    Combining Lemma \ref{Fh}, \eqref{F1}, \eqref{F2}, \eqref{F41} and \eqref{F42}, we conclude that
\begin{gather*}
    t^{1-\alpha}\|F^{(1)}(t)\|_{L^\infty(\mathbb{T}^5)}+t^{\frac{3}{2}-\alpha}\|\nabla F^{(1)}(t)\|_{C^\kappa(\mathbb{T}^5)}\leq\varepsilon_0,\quad\forall t\in(0,1]
\end{gather*}
which finishs the proof.
\end{proof}

\subsection{Semigroup of the linearization around the principal part}\label{sec:4.2}

For $i \in \{1,2\}$, $(f,g) \in Y$, and $0<t'\le t \le 1$, we define the semigroup $P^{(i)}(t,t')f$ and $Q^{(i)}(t,t')g$ as the solution of
\[
\begin{cases}
    \begin{aligned}
        (\partial_{t}-\Delta) P^{(i)}(t,t')f + \mathbb{P}\text{div}\Bigl(
        &P^{(i)}(t,t')f \otimes v^{(i)}(t) + v^{(i)}(t) \otimes P^{(i)}(t,t')f \\
        &- b^{(i)}(t)\otimes Q^{(i)}(t,t')g - Q^{(i)}(t,t')g\otimes b^{(i)}(t)\Bigr) = 0,
    \end{aligned} \\
    P^{(i)}(t',t')f = \mathbb{P}\text{div}\,f(t'), \\
    \begin{aligned}
        (\partial_{t}-\Delta) Q^{(i)}(t,t')g + \text{div}\Bigl(
        &v^{(i)}(t)\otimes Q^{(i)}(t,t')g - Q^{(i)}(t,t')g\otimes v^{(i)}(t) \\
        &+ P^{(i)}(t,t')f\otimes b^{(i)}(t) - b^{(i)}(t)\otimes P^{(i)}(t,t')f\Bigr) = 0,
    \end{aligned} \\
    Q^{(i)}(t',t')g = \text{div}\,g(t').
\end{cases}
\]
Because $v^{(i)}$ and $b^{(i)}$ are smooth with uniform estimates on $[t',1] \times \mathbb{T}^{d}$, we then have $(P^{(i)},Q^{(i)})(t,t')$ is well-defined and smooth. For convenience, we give a proof of existence on this perturbed MHD systems in Proposition \ref{lemmaExistence}. The following proposition shows that $(P^{(i)},Q^{(i)})(t,t')$ behaves comparably to $(e^{(t-t')\Delta}\mathbb{P}\mathrm{div},e^{(t-t')\Delta}\mathrm{div})$ up to a small loss of the form $\left(\frac{t}{t'}\right)^{\epsilon}$, where $\epsilon > 0$ can be chosen arbitrarily small. We note that this loss arises because $(v^{(i)},b^{(i)})$ does not belong to $(L^2(\mathbb{R}_{+};L^{\infty}))^2$ or $(L^1(\mathbb{R}_{+},t^{-\frac{1}{2}}dt;L^{\infty}))^2$.
\begin{prop}\label{prop:4.2}
    For all $i \in {1,2}$, $(a,b) \in Y$, and $0 < t' \le t \le 1$, we have
    \begin{align*}
        &\quad\| (P^{(i)}(t,t')f,Q^{(i)}(t,t')g)\|_{L^{\infty}(\mathbb{T}^5)}+(t-t')^{\frac{1}{2}}t^{\frac{\kappa}{2}}\| \nabla(P^{(i)}(t,t')f,Q^{(i)}(t,t')g)\|_{C^{\kappa}(\mathbb{T}^5)}\\& \lesssim (t')^{-1+\alpha-\epsilon}t^{-\frac{1}{2}+\epsilon}\|(f,g)\|_Y.
    \end{align*}
\end{prop}
\begin{proof}
    Without loss of generality, we write $(v,b)$ for either $(v^{(1)},b^{(1)})$ or $(v^{(2)},b^{(2)})$ and $(P(t,t'),Q(t,t'))$ for either $(P^{(1)}(t,t'),Q^{(1)}(t,t'))$ or $(P^{(2)}(t,t'),G^{(2)}(t,t'))$. For a fixed $t' \in (0,1)$, applying the Duhamel formula, we rewrite $(P(t,t'),Q(t,t'))$ and split it into three parts:
    \begin{align*}
        \begin{cases}
            P(t,t')f=e^{(t-t')\Delta}\mathbb{P}\text{div}f(t')-\int^{t}_{t'} e^{(t-s)\Delta}\mathbb{P}\text{div}\Bigl(
        P^{(i)}(t,t')f \otimes v^{(i)}(t) + v^{(i)}(t) \otimes P^{(i)}(t,t')f \\
        \quad\quad\quad\quad\quad\quad\quad\quad\quad\quad\quad\quad\quad\quad\quad\quad\quad\quad\quad\,\,- b^{(i)}(t)\otimes Q^{(i)}(t,t')g - Q^{(i)}(t,t')g\otimes b^{(i)}(t)\Bigr)ds\\
            \quad\quad\quad\quad:=I_1+II_1+III_1,\\
            Q(t,t')g=e^{(t-t')\Delta}\text{div}g(t')-\int^{t}_{t'} e^{(t-s)\Delta}\text{div}\Bigl(
        v^{(i)}(t)\otimes Q^{(i)}(t,t')g - Q^{(i)}(t,t')g\otimes v^{(i)}(t)\\
        \quad\quad\quad\quad\quad\quad\quad\quad\quad\quad\quad\quad\quad\quad\quad\quad\quad\quad+ P^{(i)}(t,t')f\otimes b^{(i)}(t) - b^{(i)}(t)\otimes P^{(i)}(t,t')f\Bigr)ds\\
            \quad\quad\quad\quad :=I_2+II_2+III_2,
        \end{cases}
    \end{align*}
    with $I_1$ and $I_2$ are the linear heat operator terms, $II_1$ and $II_2$ are the integral parts over over $[t',t'\vee (\frac{t}{2})]$ and $III_1$ and $III_2$ are the integral parts over $[t'\vee (\frac{t}{2}),t]$.
    
    We first adress $I_1$ and $I_2$. By applying Lemma \ref{2.7} to either $e^{(t-t')\Delta}\mathbb{P}$ or $e^{(t-t')\Delta}\mathbb{P}\mathrm{div}$, we choose the more favorable upper bound and have
    \begin{align}
        \| (I_1,I_2)\|_{L^{\infty}} &\lesssim (\| \nabla f(t')\|_{C^{\kappa}}+\| \nabla g(t')\|_{C^{\kappa}})\wedge((t-t')^{-\frac{1}{2}}(\| f(t')\|_{L^{\infty}}+\| g(t')\|_{L^{\infty}})\nonumber
        \\&\lesssim (t')^{-\frac{3}{2}+\alpha}\wedge ((t-t')^{-\frac{1}{2}}(t')^{-1+\alpha})\| (f,g)\|_{Y}\nonumber
        \\&\lesssim(t')^{-1+\alpha}(t'\vee (t-t'))^{-\frac{1}{2}}\| (f,g)\|_{Y}\nonumber
        \\&\lesssim(t')^{-1+\alpha}t^{-\frac{1}{2}}\| (f,g)\|_{Y}.\label{I1}
    \end{align}
    
    Next, we focus on $II_1$ and $II_2$. Combining Lemma \ref{2.7} and the fact that the integral vanishes identically unless $t>2t'$,  we obtain
    \begin{align}
        \| (II_1,II_2)\|_{L^{\infty}} &\lesssim \int^{t'\vee(\frac{t}{2})}_{t'} (t-s)^{-\frac{1}{2}}\| (v(s),b(s))\|_{L^{\infty}}\| (P(s,t')f,Q(s,t')g)\|_{L^{\infty}}ds\nonumber
        \\& \lesssim t^{-\frac{1}{2}} \int^{t'\vee(\frac{t}{2})}_{t'} (t-s)^{-\frac{1}{2}}\| (v(s),b(s))\|_{L^{\infty}}\| (P(s,t')f,Q(s,t')g)\|_{L^{\infty}}s^{\frac{1}{2}}ds.\label{II2}
    \end{align}
    
    For $III_1$ and $III_2$, by Lemma \ref{2.7} and the fact that $(\frac{t}{s})^\frac{1}{2} \le \sqrt{2}$ on the interval $s \in [t'\vee (\frac{t}{2}),t]$, it is easy to see that
    \begin{align}       
        \| (III_1,III_2)\|_{L^{\infty}} &\lesssim \int^{t}_{t'\vee(\frac{t}{2})} (t-s)^{-\frac{1}{2}}\| (v(s),b(s))\|_{L^{\infty}}\| (P(s,t')f.Q(s,t')g)\|_{L^{\infty}}ds\nonumber
        \\&\lesssim t^{-\frac{1}{2}} \int^{t}_{t'\vee(\frac{t}{2})}(t-s)^{-\frac{1}{2}}\| (v(s),b(s))\|_{L^{\infty}}\| (P(s,t')f.Q(s,t')g)\|_{L^{\infty}}s^{\frac{1}{2}}ds.\label{III3}
    \end{align}
    Letting $h(t):=t^{\frac{1}{2}}\| (P(s,t')f,Q(s,t')g)\|_{L^{\infty}}$ and summing the estimates from \eqref{I1}--\eqref{III3}, we deduce that
    \begin{align*}
        h(t) \lesssim (t')^{-1+\alpha}\|(f,g)\|_{Y}+\int^{t}_{t'} (s^{-\frac{1}{2}}+(t-s)^{-\frac{1}{2}})\|(v(s),b(s)\|_{L^{\infty}}h(s)ds.
    \end{align*}
    Now using Lemma \ref{B.3} with $p=3$, $g_1(s)=s^{-\frac{1}{2}}\|(v(s),b(s))\|_{L^{\infty}}$ and $g_2(s)=\|(v(s),b(s))\|_{L^{\infty}}$, we then have
    \begin{align*}
        h(t) &\lesssim (t')^{-1+\alpha}\|(f,g)\|_{Y}\exp(O(\| (v,b)\|_{L^1([t',t],t^{-\frac{1}{2}}dt;L^{\infty})}+\| s^{\frac{1}{2}}(v,b)\|_{L^{\infty}_{t,x}([t',t])}\| (v,b)\|_{L^2([t',t];L^{\infty})}))\\
        &\lesssim(t')^{-1+\alpha}\|(f,g)\|_{Y}\exp(O(1+(\log A)^{-1}\log(\frac{t}{t'}))).
    \end{align*}
    where \eqref{1 estimate1 of v} and \eqref{2 estimate of v} imply the second inequality.
    Choosing $A$ sufficiently large depending on $\epsilon$, the power becomes smaller than $\frac{\epsilon}{2}$ and we conclude that
    \begin{align}\label{3.13}
        \| (P(t,t')f,Q(t,t')g)\|_{L^{\infty}} \lesssim t^{-\frac{1}{2}}(t')^{-1+\alpha}(\frac{t}{t'})^{\frac{\epsilon}{2}}\| (f,g)\|_{Y}.
    \end{align}
    
    Next, we consider $C^{1,k}$ norm.  $I_1$ and $I_2$ stays the same but $II_1$ and $II_2$ split the integral into $[t',\frac{(t+t')}{2}]$, and $III_1$ and $III_2$ split the integral into $[\frac{(t+t')}{2},t]$. 
    
    Similarly, using Lemma \ref{2.7} and the fact that $t'\vee(t-t')^{1+\kappa}\ge(\frac{t}{2})^{1+\kappa}$ and $t\le 1$ give
    \begin{align}
        \| \nabla(I_1,I_2)\|_{C^{\kappa}} &\lesssim ((t-t')^{-\frac{1}{2}}\| \nabla(f(t'),g(t'))\|_{C^{\kappa}})\wedge((t-t')^{-1-\frac{\kappa}{2}}\|(f(t'),g(t'))\|_{L^{\infty}})\nonumber
        \\&\lesssim ((t-t')^{-\frac{1}{2}}(t')^{-\frac{3}{2}+\alpha})\wedge((t-t')^{-1-\frac{\kappa}{2}}(t')^{-1+\alpha})\|(f,g)\|_{Y}\nonumber
        \\&\lesssim (t-t')^{-\frac{1}{2}}(t')^{-1+\alpha}t^{-\frac{1+\kappa}{2}}\|(f,g)\|_{Y}.\label{Ia}
    \end{align}
    
    For $\nabla II_1$ and $\nabla II_2$, together with Lemma \ref{2.7}, \eqref{1 estimate1 of v} and \eqref{3.13}, we get
    \begin{align}
        \| \nabla(II_1,II_2)\|_{C^{\kappa}} &\lesssim \int^{\frac{(t+t')}{2}}_{t'} (t-s)^{-1-\frac{\kappa}{2}}\| (f(s),g(s))\|_{L^{\infty}}\|(P(s,t')f,Q(s,t')g)\|_{L^{\infty}}ds\nonumber
        \\&\lesssim (t')^{-1+\alpha-\frac{\epsilon}{2}}\int^{\frac{(t+t')}{2}}_{t'}(t-s)^{-1-\frac{\kappa}{2}}s^{-1+\frac{\epsilon}{2}}\|(f,g)\|_{Y}ds\nonumber
        \\&\lesssim (t-t')^{-\frac{\kappa}{2}}(t')^{-1+\alpha-\frac{\epsilon}{2}}t^{-1+\frac{\epsilon}{2}}\|(f,g)\|_{Y}.\label{IIb}
    \end{align}
    
    By \eqref{3.13} once again and the fact that $C^{1,\kappa}\cap L^{\infty}$ is a multiplication algebra, it holds that
\begin{align}
&\quad\| \nabla(III_1,III_2)\|_{C^{\kappa}} \nonumber\\
&\lesssim \int_{\frac{t+t'}{2}}^{t} (t-s)^{-\frac{1}{2}} \Bigl(
    \| \nabla(v(s),b(s))\|_{C^{\kappa}} \| (P(s,t')f,Q(s,t')g)\|_{L^{\infty}} \nonumber
\\&\quad\quad\quad\quad\quad\quad\quad\quad + \| (v(s),b(s))\|_{L^{\infty}} \| \nabla(P(s,t')f,Q(s,t')g)\|_{C^{\kappa}}
\Bigr) ds\nonumber \\
&\lesssim \int_{\frac{t+t'}{2}}^{t} (t-s)^{-\frac{1}{2}} \Bigl(
    s^{-\frac{3}{2}+\frac{\epsilon}{2}-\frac{\kappa}{2}} (t')^{-1+\alpha-\frac{\epsilon}{2}} \|(f,g)\|_{Y}  + \| (v(s),b(s))\|_{L^{\infty}} \| \nabla(P(s,t')f,Q(s,t')g)\|_{C^{\kappa}}
\Bigr) ds\nonumber\\
&\lesssim(t')^{-1+\alpha-\frac{\epsilon}{2}}t^{-1+\frac{\epsilon}{2}-\frac{\kappa}{2}}\|(f,g)\|_{Y}+\int_{\frac{t+t'}{2}}^{t} (t-s)^{-\frac{1}{2}}\| (v(s),b(s))\|_{L^{\infty}} \| \nabla(P(s,t')f,Q(s,t')g)\|_{C^{\kappa}}
\Bigr) ds\label{Ic}
\end{align}
    
     Hence, defining $h(t):=(t-t')^{\frac{1}{2}}t^{\frac{1+\kappa}{2}}\| \nabla(P(t,t'),Q(t,t'))\|_{C^{\kappa}}$ and summing \eqref{Ia}--\eqref{Ic},
    we then have
    \begin{align*}
        h(t)\lesssim (t')^{-1+\alpha}(\frac{t}{t'})^{\frac{\epsilon}{2}}\|(f,g)\|_{Y}+\int^{t}_{\frac{t+t'}{2}} (t-s)^{-\frac{1}{2}}\|(f(s),g(s))\|_{L^{\infty}}h(s)ds.
    \end{align*}
    Finally, we conclude once again by Lemma \ref{B.3}, \eqref{1 estimate1 of v} and \eqref{2 estimate of v} to gain the desired result.
\end{proof}

\subsection{Fixed point argument}\label{sec:4.3}
In Proposition \ref{Prop 4.1}, we have showed that the subcritical residual terms $\{(F^{(i)},G^{(i)})\}_{i\in\{1,2\}}$ can be arbitrary small. Combining this fact with the zero initial data, we can construct the perturbation $\{(w^{(i)},\rho^{(i)})\}_{i\in\{1,2\}}$ by the fixed point theorem. 

Let us introduce the critical sapce for the perturbation $\{(w^{(i)},\rho^{(i)})\}_{i\in\{1,2\}}$:
\begin{align*}
    X=\{(\omega,\rho)\in(C^0((0,1];C^{1,\kappa}(\mathbb{T}^5;\mathbb{R}^5))^2\},
\end{align*}
with norm
\begin{align*}
    \| (\omega,\rho)\|_{X}=\sup_{t\in (0,1]}(t^{\frac{1}{2}-\frac{\alpha}{2}}\|(\omega,\rho)\|_{L^{\infty}}+t^{1-\frac{\alpha}{2}}\|\nabla (\omega,\rho)\|_{C^{\kappa}})<\infty.
\end{align*}

\begin{prop}\label{prop:4.3}
    Let $A$ be sufficiently large.
    There exists $C'>0$ such that for all $\epsilon'\in (0,C'^{-1})$ and $i\in\{1,2\}$, there exist $(\omega^{(i)},\rho^{(i)})\in B_X(0,\epsilon')$ such that 
    \begin{align*}
        \begin{cases}
            \tilde{v}^{(i)}=v^{(i)}+\omega^{(i)},\\
            \tilde{b}^{(i)}=b^{(1)}+\rho^{(i)},
        \end{cases}
    \end{align*}
    is solution of (\ref{e:MHDe}).
    Furthermore, $(\omega^{(i)},\rho^{(i)}) \to (0,0)$ in $(C^{(-1+\frac{\alpha}{2})})^2$ as $t\to 0$.
\end{prop}
\begin{proof}
    Through this proof, we suppress the dependence on $i$ of $(v^{(i)},b^{(i)})$, $(\omega^{(i)},\rho^{(i)})$ and $(F^{(i)},G^{(i)})$.

    A necessary condition for $(\tilde{v},\tilde{b})=(v+\omega,\tilde{v}+\rho)$ to solve the \eqref{e:MHDe} with initial data $(v(0),b(0))$ is that the perturbation $(\omega,\rho)$ must satisfy
    \begin{align*}
        \begin{cases}
            (\partial_{t}-\Delta)\omega+\mathbb{P}\text{div}\, (v\otimes \omega+\omega\otimes v-b\otimes \rho- \rho\otimes b)=\mathbb{P}\text{div}\, (F-\omega\otimes \omega+\rho \otimes \rho),\\
            (\partial_{t}-\Delta)\rho+\text{div}\,(v\otimes \rho-\rho\otimes v+\omega\otimes b-b\otimes \omega)=\text{div}\, (G-\omega\otimes \rho+\rho\otimes \omega),\\
            \rho(0,\cdot)=\omega(0,\cdot)=0.
        \end{cases}
    \end{align*}
    where $(F,G)$ is as in Proposition \ref{Prop 4.1}. Recalling that $(P(t,t'),Q(t,t'))$ is the semigroup of the linearization around $(v,b)$ defined in Subsection \ref{sec:4.2}, we solve for $(\omega,\rho)$ as a fixed point of the $Y\rightarrow X$ map $(\omega,\rho)\mapsto(T^1(\omega),T^2(\rho))$:
    \begin{align*}
        (\omega,\rho)\mapsto \begin{pmatrix}
        T^1(\omega)(t)\\T^2(\rho)(t)
    \end{pmatrix}
        =\begin{pmatrix}
        \int^{t}_{0} P(t,t')(F-\omega\otimes \omega+\rho\otimes \rho)(t')dt'\\\int^t_0 Q(t,t')(G-\omega\otimes \rho+\rho\otimes \omega)(t')dt'
    \end{pmatrix}
    \end{align*}
    
     Firstly, we prove that $(T_1,T_2)$ is a well-defined operator on $X$.
    Due to the definition of $X$ and $Y$, we have the elementary product estimate
    \begin{align*}
        \|\omega\otimes\omega\|_{Y}\lesssim \sup_{0<t\le1}(t^{\frac{3}{2}-\alpha}\|\omega(t)\|_{L^{\infty}}\|\nabla \omega(t)\|_{C^{\kappa}}+t^{1-\alpha}\|\omega(t)\|^2_{L^{\infty}})\lesssim\|(\omega,\rho)\|^2_{X}.
    \end{align*}
    We then similarly have $\|\rho\otimes \rho\|_{Y},\|\omega\otimes\rho\|_{Y},\|\rho\otimes \omega\|_{Y}\lesssim \|(\omega,\rho)\|_{X}$.
    Combining this with Propostion \ref{Prop 4.1} and \ref{prop:4.2} and the fact that $\epsilon < \alpha$,
    \begin{align*}
        \|(T^1(\omega)(t),T^2(\rho)(t))\|_{L^{\infty}} &\lesssim t^{-\frac{1}{2}+\epsilon}\int^t_0 (t')^{-1+\alpha-\epsilon}dt'(\|(\omega,\rho)\|^2_X+\|(F,G)\|_Y)
        \\& \lesssim t^{-\frac{1}{2}+\alpha}(\|(\omega,\rho)\|_X^2+\epsilon_0).
    \end{align*}
    and 
    \begin{align*}
        \|\nabla(T^1(\omega)(t),T^2(\rho)(t))\|_{C^{\kappa}} &\lesssim t^{-\frac{1}{2}+\epsilon}\int^t_0 (t')^{-1+\alpha-\epsilon}(t-t')^{-\frac{1}{2}}dt'(\|(\omega,\rho)\|^2_X+\|(F,G)\|_Y)
        \\& \lesssim t^{-1+\alpha}t^{-\frac{\kappa}{2}}(\|(\omega,\rho)\|_X^2+\epsilon_0).
    \end{align*}
    We let $\epsilon_0$ in Proposition \eqref{Prop 4.1} be sufficiently small and then we get 
    \begin{align*}
        \|(T^{1}(\omega)(t),T^2(\rho)(t))\|_{X}\le O(\|(w,\rho)\|_{X})+\frac{1}{2}\epsilon',
    \end{align*}
    having also used $t\le 1$. By a similar calculation,
    \begin{align*}
            \| (T^1(\omega_1)-T^1(\omega_2),T^2(\rho_1)-T^2(\rho_2))\|_X\le \|(\omega_1-\omega_2,\rho_1-\rho_2)\|_X(\|(\omega_1,\rho_1)\|_{X}+\|(w_2,\rho_2)\|_X).
    \end{align*}
    Choosing $C'$ sufficiently large to compensate for the implicit constants, it follows that $T$ is a contraction on the $B_X(0,\epsilon')$ for all $\epsilon' \in (0,C'^{-1})$, and we conclude by the banach fix point theorem.
    
    Finally, we prove control of $\omega$ and $\rho$ near the initial time using the Duhamel formula
    \begin{align}
        \begin{cases}
            \omega(t)=\int^{t}_{0} e^{(t-t')\Delta}\mathbb{P}\mathrm{div}\,(F-\omega\otimes \omega+\rho \otimes\rho-v\otimes \omega-\omega\otimes v+b\otimes \rho+\rho\otimes b)(t')dt',\\
            \rho(t)=\int^{t}_{0} e^{(t-t')\Delta}\text{div}\,(G-\omega\otimes \rho+\rho\otimes \omega-v\otimes \rho+\rho\otimes v-\omega\otimes b+b\otimes \omega)(t')dt',
        \end{cases}\label{czong}
    \end{align}
By \eqref{2.7}, \eqref{1 estimate1 of v}, Proposition \ref{Prop 4.1}, the interpolation inequality and the fact that $\|(w,\rho)\|_X \le \epsilon' $, we arrive at
\begin{align}
    \| (F,G)\|_{C^{\frac{\alpha}{2}}}\lesssim (t')^{-1+\frac{3\alpha}{4}},\label{c1}
\end{align}
\begin{align}
    \|(\omega,\rho)\otimes (\omega,\rho)\|_{C^{\frac{\alpha}{2}}}\lesssim (t')^{-1+\frac{3\alpha}{4}}\|(\omega,\rho)\otimes(\omega,\rho)\|_{Y}\lesssim   (t')^{-1+\frac{3\alpha}{4}}\|(\omega,\rho)\|_{X}^2\lesssim (t')^{-1+\frac{3\alpha}{4}} \label{c2}
\end{align}
and
\begin{align}
    \|(v,b)\otimes(w,\rho)\|_{C^{\frac{\alpha}{2}}} &\lesssim \|(v,b)\|_{C^{\frac{\alpha}{2}}}\|(w,\rho)\|_{L^{\infty}}+\|(v,b)\|_{L^{\infty}}\|(\omega,\rho)\|_{C^{\frac{\alpha}{2}}}\nonumber
    \\&\lesssim (t')^{-\frac{1}{2}-\frac{\alpha}{4}}(t')^{\frac{\alpha}{2}-\frac{1}{2}}+(t')^{-\frac{1}{2}}(t')^{\frac{\alpha}{4}-\frac{1}{2}}\nonumber
    \\&\lesssim (t')^{-1+\frac{\alpha}{4}}\label{c3}
\end{align}
Plugging \eqref{c1}, \eqref{c2} and \eqref{c3} into (\ref{czong}), we obtain
\begin{align*}
    \|(\omega(t),\rho(t)\|_{C^{-1+\frac{\alpha}{2}}} \lesssim \int^{t}_{0} ((t')^{-1+\frac{\alpha}{2}} +(t')^{-1+\frac{3\alpha}{4}})dt'\lesssim t^{\frac{\alpha}{4}},\quad\forall t\in (0,1],
\end{align*}
which implies that the pair $(\omega(t),\rho(t)$ tends to $0$ in ${C^{-1+\frac{\alpha}{2}}}$ as $t\rightarrow0$. 

\end{proof}

\section{Proof of the non-uniqueness theorem}\label{Proof of the Non-Uniqueness}
Recall that
    \begin{align*}
        \begin{cases}
            v^{(1)}\triangleq\sum_{k\geq0\,even}\mathrm{curl}\mathrm{curl}(\psi_k+\theta_k)+\sum_{k\geq0\,odd}\bar{v}_k,\\b^{(1)}\triangleq\sum_{k\geq0\,even}\mathrm{curl}\mathrm{curl}\,\phi_k+\sum_{k\geq0\,odd}\bar{b}_k ,
        \end{cases}
    \end{align*}
    and 
    \begin{align*}
    \begin{cases}
        v^{(2)}\triangleq\sum_{k\geq0\,odd}\mathrm{curl}\mathrm{curl}(\psi_k+\theta_k)+\sum_{k\geq0\,even}\bar{v}_k,\\ b^{(2)}\triangleq\sum_{k\geq0\,odd}\mathrm{curl}\mathrm{curl}\,\phi_k^0+\sum_{k\geq0\,even}\bar{b}_k. 
    \end{cases} 
    \end{align*}

In Section \ref{sec:4.3}, we have shown that for $i \in \{1,2\}$, there exist $(\omega^{(i)},\rho^{(i)}) \in C^0((0,1];C^{1,\kappa})$ obeying suitable equations such that $(\tilde{v}^{(i)},\tilde{b}^{(i)})=(v^{(i)}+\omega^{(i)},b^{(i)}+\rho^{(i)})$ obey \eqref{e:MHDe}. Moreover, by standard regularity theory, $(\tilde{v}^{(i)},\tilde{b}^{(i)})$ is smooth on $(0,1]\times \mathbb{T}^5$.

Next we argue that both solutions attain the initial data $(U^0,B^0)$. As shown in Proposition \ref{prop:4.3}, $(\omega^{(i)},\rho^{(i)}) \to 0$ in $C^{-1+\frac{\alpha}{2}}(\mathbb{T}^5)$. However, this subcritical topology is too stong for the principal part of the solutions. Instead we claim that for all $1<p<\infty$, $(v^{(i)},b^{(i)})\to (U^0,B^0)$ in $\dot{W}^{-1,p}(\mathbb{T}^5)$.
Therefore, for the heat-dominated part of the data, we use Definition \ref{defi of v} and Proposition \ref{prop nablampsi} to estimate
\begin{align*}
    &\quad\quad\| \sum_{k\geq 0, even}(\mathrm{curl}(\psi_k+\theta_k)-\mathrm{curl}(\psi_k^0+\theta_k^0))\|_{L^p}\\
    &\lesssim\| \sum_{k\geq 0, even}(\mathrm{curl}(\psi_k^0+\theta_k^0)\mathrm{exp}(-N_k^2t)-\mathrm{curl}(\psi_k^0+\theta_k^0))\|_{L^p} \\
    &\lesssim \left(\sum_{k,\xi\in\Lambda_U}\|\mathrm{curl}\,(a_{\xi,k}\Psi^0_{\xi,k})\|_{L^p}+\sum_{k,\xi\in\Lambda_B}\|\mathrm{curl}\,(b_{\xi,k}\Theta^0_{\xi,k})\|_{L^p}\right)(1-e^{-N^2_kt})
    \\&\lesssim \sum_{k} |{\Omega}_k^3|^{\frac{1}{p}}(1\wedge(N_k^2t))
    \\&\lesssim \sum_{k:N_k\le t^{-\frac{1}{4}}} N_k^2t+\sum_{k:N_k>t^{-\frac{1}{4}}}2^{-\frac{k}{p}}.   
\end{align*}
By the same reasoning as before, we also obtain 
\begin{align*}
    \| \sum_{k\geq 0, even}(\mathrm{curl}\,\phi_k-\mathrm{curl}\,\phi_k^0)\|_{L^p}\lesssim \sum_{k:N_k\le t^{-\frac{1}{4}}} N_k^2t+\sum_{k:N_k>t^{-\frac{1}{4}}}2^{-\frac{k}{p}}.  
\end{align*}
Since the series in the right is finite and tends to 0 as $t\to0$, we conclude that
\begin{align*}
    (\mathrm{curl}\sum_{k,even}\mathrm{curl}\,(\psi_k+\theta_k)(t),\mathrm{curl}\sum_{k,even}\mathrm{curl}\,\phi_k(t))\underset{\dot{W}^{-1,p}}{\to}(\mathrm{curl}\sum_{k\in N_i}\mathrm{curl}\,(\psi_k^0+\theta_k^0),\mathrm{curl}\sum_{k\in N_i}\mathrm{curl}\,\phi^0_k).
\end{align*}

 For the inverse cascade-dominated part $\bar{v}_k(t)$, with Definition \ref{defi of v}, we split it into two parts:
\begin{align*}
    \mathrm{curl}\mathrm{curl}(\psi_k^0+\theta_k^0)-\bar{v}_k(t)=I_1+I_2,
\end{align*}
with
\begin{align*}
        I_1\triangleq&\frac{1}{2}N_{k+1}^{-2}\mathbb{P}\mathrm{div}\left(\sum_{\xi\in\Lambda_U}A_{\xi,k+1}a_{\xi,k+1}^2\xi_1\otimes\xi_1+\sum_{\xi\in\Lambda_B}B_{\xi,k+1}b_{\xi,k+1}^2(\xi_2\otimes\xi_2-\xi_1\otimes\xi_1)\right)\\&\times(1-\mathrm{exp}(-2N_{k+1}^2t))
\end{align*}
and 
\begin{align*}
        I_2\triangleq &\mathrm{curl}\mathrm{curl}\,(\psi_k^0+\theta_k^0)\\
        &-\frac{1}{2}N_{k+1}^{-2}\mathbb{P}\mathrm{div}\left(\sum_{\xi\in\Lambda_U}A_{\xi,k+1}a_{\xi,k+1}^2\xi_1\otimes\xi_1+\sum_{\xi\in\Lambda_B}B_{\xi,k+1}b_{\xi,k+1}^2(\xi_2\otimes\xi_2-\xi_1\otimes\xi_1)\right).
\end{align*}
By the same computation from Remark \ref{intialequal}, $I_2$ is identically zero. For $I_1$, using Propostion \ref{prop nablampsi} and \eqref{e Omegak}, we arrive at
\begin{align*}
    \|I_1\|_{\dot{W}^{-1,p}}&\lesssim \sum_{k,\xi\in\Lambda_U}(1-e^{-2N^2_{k+1}t})N^{-2}_{k+1}\|a_{\xi,k+1}^2\|_{L^p}+\sum_{k,\xi\in\Lambda_B}(1-e^{-2N^2_{k+1}t})N^{-2}_{k+1}\|b_{\xi,k+1}^2\|_{L^p}\\
    &\lesssim\sum_k(1\wedge(N^2_{k+1}t))|\Omega_k^3|^{\frac{1}{p}}.
\end{align*}
Similarly, it is easy to see that 
\begin{align*}
    \|\sum_{k\geq0\,odd} ( \mathrm{curl}\mathrm{curl}\,\phi_k^0-\bar{b}_k(t))\|_{\dot{W}^{-1,p}}\lesssim \sum_{k}(1\wedge(N^2_{k+1}t))|\Omega_k^3|^{\frac{1}{p}}.
\end{align*}
The previous series in the right is finite and also tends to 0 as $t\to0$, we then complete the proof that the initial data is attained by both solutions. Moreover, we prove that $\{(\tilde{v}^{(i)},\tilde{b}^{(i)})\}$ are belong to $C^0([0,\infty);\dot{W}^{-1,p}(\mathbb{T}^5))$.

Finnally, we show that the two solutions are distinct. We consider a time scale when all but the lowest frequency mode should have dissipated away, for instance $t_0=N_0^{-2}$. Recall from the Proposition \ref{prop:4.3} that $(\omega^{(i)},\rho^{(i)})\in B_{X}(0,\epsilon')$. By the triangle inquality, we get
\begin{align}
&\quad\|(v^{(1)}(t_0)-v^{(2)}(t_0),b^{(1)}(t_0)-b^{(2)}(t_0))\|_{L^\infty}\nonumber
\\&\ge \|(v_0,b_0)\|_{L^\infty} - \sum_{k\ge 1}\|(v_k(t_0),b_k(t_0))\|_{L^\infty}  - \sum_{k\ge 0}\|(\bar{v}_k(t_0),\bar{b}_k(t_0))\|_{L^\infty} \nonumber\\&
 - \|(w^{(1)}(t_0),\rho^{(1)}(t_0))\|_{L^\infty}  - \|(w^{(2)}(t_0),\rho^{(2)}(t_0))\|_{L^\infty}.\label{all}
\end{align}
By Lemma \ref{lemma vk}, one obtains 
\begin{align}
    \sum_{k\ge 1}\|(v_k(t_0),b_k(t_0))\|_{L^\infty}&\lesssim \sum_{k\ge 1}N_k\exp{(-N_k^2N_0^{-2})}\lesssim N_1\exp{(-N_1^2N_0^{-2})}
   \lesssim N_0\left(\frac{N_1}{N_0}\right)^{-100},\label{y1}
\end{align}
and
\begin{align}
    \sum_{k\ge 0}\|(\bar{v}_k(t_0),\bar{b}_k(t_0))\|_{L^{\infty}}\lesssim\sum_{k\ge 0}N_k\exp{(-N_{k+1}^2N_0^{-2})}\lesssim N_0\exp{(-N_1^2N_0^{-2})}\lesssim N_0\left(\frac{N_1}{N_0}\right)^{-100}.\label{y2}
\end{align}
 For the $\{(\omega^{(i)},\rho^{(i)})\}_{i\in\{1,2\}}$, we have by Proposition \ref{prop:4.3}
\begin{align}
    \|(\omega^{(i)}(t_0),\rho^{(i)}(t_0))\|_{L^{\infty}}\lesssim \epsilon't_0^{-\frac{1}{2}+\frac{\alpha}{2}}=\epsilon'N_0^{1-\alpha}=\epsilon'A^{-\alpha\gamma}N_0.\label{y3}
\end{align}
Recall the notation of Lemma \ref{lemmvp} and \eqref{e vek}, we have the decomposition $(v_0,b_0)=(v_0^p+v_0^e,b_0^p+b_0^e)$ and the estimate
\begin{align}
    \|(v_0^e,b_0^e)\|_{L^{\infty}}\lesssim M_0 \lesssim 2N_0A^{-\gamma+1}.\label{y4}
\end{align}
We now only need to establish the lower bound on $(v_0^p(t,x),b_0^e(t,x))$, recalling from Definition \ref{defi of psiok} that $(a_{\xi,0},b_{\xi',0})$ takes a particularly simple form when $k=0$. By inspection of Definition \ref{defi potential}, we easily have $\Psi_{1,0}\gtrsim N_0^{-2}$, $\Phi_{1,0}\gtrsim N_0^{-2}$ and $\Theta_{1,0}\gtrsim N_0^{-2}$, which yields 
\begin{align}
    \|(v_0^p(t_0),b^p_0(t_0))\|_{L^{\infty}}\ge C^{-1}N_0,\label{y5}
\end{align}
 where $C$ is a large enough constant. Plugging \eqref{y1}-\eqref{y5} into (\ref{all}), we have
\begin{align*}
    \|(v^{(1)}(t_0)-v^{(2)}(t_0),b^{(1)}(t_0)-b^{(2)}(t_0))\|_{L^\infty}\ge C^{-1}N_0-A^{-c(b,\alpha,\lambda)}N_0,
\end{align*}
for a $c(b,\alpha,\lambda)>0$. With a sufficiently large choice of $A$, we conclude that $(v^{(1)},b^{(1)})$ and $(v^{(2)},b^{(2)})$ are distinct.

Finally, extend our solutions on $[0,1]$ to be global solutions. Recall from \eqref{1 estimate1 of v} that $\|\nabla^m(v^{(i)},b^{(i)})|_{t=1}\|_{L^{\infty}}\le \exp{(-N_0^2/O_m(1))}$ for all $m\ge 0$. Then $\|(P_Nv^{(i)},P_Nb^{(i)})|_{t=1}\|_{L^{\infty}}\lesssim \langle N\rangle^{-5}\exp{-A^{2\lambda}}$ and therefore $(v^{(i)},b^{(i)})|_{t=1}$ can be made arbitrarily small in, say, $H^{-\frac{1}{2}}(\mathbb{T})$ or $B^{-1}_{\infty,2}\subset BMO^{-1}$, and existence of a global-in-time solution follows.

\appendix

\section{Tools from convex integration }
In this section, we recall some useful tools which is developed in the convex integration method. 

\begin{lemm}\label{lemmaNS}
    There exists a set $\Lambda_U \subset \mathbb{S}^{d-1} \cap \mathbb\,{\mathbb{Q}}^d$ that consists of vectors $\xi$ with associated orthonormal bases $(\xi,\xi_1,  \cdots, \xi_{d-1})$, $\varepsilon_0 > 0$, $C>0$ and smooth positive functions $\Gamma_{\xi}: B_{\varepsilon}(\mathrm{Id}) \to \mathbb{R}$, where $B_{\varepsilon_0}(\mathrm{Id})$ is the ball of radius $\varepsilon_0$ centered at $\mathrm{Id}$ in the space of $d \times d$ symmetric matrices, such that we have the following identity:
\begin{gather}
    S = \sum_{\xi \in \Lambda_U} \Gamma_{\xi}^2(S) \xi_1 \otimes \xi_1 ,\quad\forall S\in B_{\varepsilon_0}(\mathrm{Id}),\nonumber\\
    \frac{1}{C}\leq\Gamma_{\xi}(S)\leq C,\quad\forall S\in B_{\varepsilon_0}(\mathrm{Id}).\label{Gamma xiajie}
\end{gather}
\end{lemm}

\begin{lemm}\label{lemmaMHD}
    There exists a set $\Lambda_B \subset \mathbb{S}^{d-1} \cap \mathbb\,{\mathbb{Q}}^d$ that consists of vectors $\xi$ with associated orthonormal bases $(\xi,\xi_1, \xi_2, \cdots, \xi_{d-1})$, $\varepsilon_0 > 0$, $C>0$ and smooth positive functions $\gamma_{\xi}: \tilde{B}_{\varepsilon_0}(0) \to \mathbb{R}$, where $\tilde{B}_{\varepsilon_0}(0)$ is the ball of radius $\varepsilon_0$ centered at $0$ in the space of $d \times d$ skew-symmetric matrices, such that we have the following identity:
\begin{gather}
    A = \sum_{\xi \in \Lambda_B} \gamma_{\xi}^2(A) \left(\xi_2 \otimes \xi_1 - \xi_1 \otimes \xi_2\right),\quad\forall A \in \tilde{B}_{\varepsilon_0}(0),\nonumber\\
    \frac{1}{C}\leq\gamma_{\xi}(A)\leq C,\quad\forall A \in \tilde{B}_{\varepsilon_0}(0).\label{gamma jie}
\end{gather}
\end{lemm}

\noindent
Let us emphasize that the universal constant $\varepsilon_0$ and $C$ can be same in Lemma \ref{lemmaNS} and Lemma \ref{lemmaMHD}.

We recall two types of antidivergence operators $\mathcal{R}^V$ and  $\mathcal{R}^B$ which are introduced in \cite{D1} and \cite{GeolemmaMHD} respectively. 
\begin{prop}\label{def of antidiv}
    There exist two linear operators $\mathcal{R}^V:C^\infty(\mathbb{T}^5;\mathbb{R}^d)\rightarrow C^\infty(\mathbb{T}^5;S^{d\times d})$ and $\mathcal{R}^B:C^\infty(\mathbb{T}^5;\mathbb{R}^d)\rightarrow C^\infty(\mathbb{T}^5;A^{d\times d})$ satisfying
    \begin{align*}
        \text{div}\,\mathcal{R}^Vf&=f-\int_{\mathbb{T}^5}f,\quad\forall f\in C^\infty(\mathbb{T}^5;\mathbb{R}^d),\\
        \text{div}\,\mathcal{R}^Bg&=g-\int_{\mathbb{T}^5}g,\quad\forall g\in C^\infty(\mathbb{T}^5;\mathbb{R}^d).
    \end{align*}
\end{prop}

\begin{lemm}
    Let $\beta>0$, for any smooth function $f$ and $m>0$, we have
    \begin{align}
        \|\mathcal{R}^V(f(x)e^{i\lambda kx})\|_{C^{\beta}}\lesssim_{\beta,m}\lambda^{-1+\beta}\|f\|_{L^\infty}+ \lambda^{-m+\beta}\|\nabla^mf\|_{L^\infty}+\lambda^{-m}\|\nabla^mf\|_{C^{\beta}},\\
         \|\mathcal{R}^B(f(x)e^{i\lambda kx})\|_{C^{\beta}}\lesssim_{\beta,m}\lambda^{-1+\beta}\|f\|_{L^\infty}+ \lambda^{-m+\beta}\|\nabla^mf\|_{L^\infty}+\lambda^{-m}\|\nabla^mf\|_{C^{\beta}}.\label{RBestimate}
    \end{align}
\end{lemm}

Finally, we recall the improved H\"{o}lder's inequality from \cite[Lemma 2.1]{modena} (see also \cite[Lemma3.7]{ns有限能量不唯一}).
\begin{lemm}\label{improved holder}
    Let $p\in[1,\infty]$ and $f,g:\mathbb{T}^5\rightarrow\mathbb{R}$ be smooth functions. Then for any $\sigma\in\mathbb{N}$, we have
    \begin{gather*}
        \|g\cdot f(\sigma\cdot)\|_{L^p}\lesssim\|g\|_{L^p}\|f\|_{L^p}+\sigma^{-\frac{1}{p}}\|g\|_{C^1}\|f\|_{L^p}.
    \end{gather*}
\end{lemm}

\section{Existence for the perturbed MHD systems}
\begin{prop} \label{lemmaExistence}
    Fix $t_0 \in \mathbb{R}$. Let $(v,b) \in (C^{\infty}([t_0,\infty)\times T^5;\mathbb{R}^5))^2$ and $(\omega^0,\rho^0)\in (C^{\infty}(\mathbb{T}^5;\mathbb{R}^5))^2$ be divergence-free vector fields. The there exists a solution $(w,\rho) \in (C^{\infty}([t_0,\infty)\times \mathbb{T}^5;\mathbb{R}^5))^2$ of 
\begin{equation}\label{B.1}
    \begin{cases}
        \partial_t \omega-\Delta \omega+\mathrm{div}(v\otimes \omega+\omega\otimes v-b\otimes \rho-\rho\otimes b) +\nabla p=0 ,\\
        \partial_t \rho-\Delta \rho+\mathrm{div}(v\otimes \rho-\rho\otimes v+\omega\otimes b-b\otimes \omega)=0,\\
        \mathrm{div}\, \omega=\mathrm{div}\, \rho =0,\\
        \omega(t_0,x)=\omega^0(x),\quad \rho(t_0,x)=\rho^0(x).
    \end{cases}
\end{equation}
\end{prop}
\begin{proof}
    It is suffice to show that the proposition on the finite interval $[t_0,T]$ for all $T>0$. We begin by proving the existence of a weak solution in $(L^{\infty}_tL^2_x\cap L^2_tH^1_x([t_0,T]\times \mathbb{T}^5;\mathbb{R}^5))^2$ by the Galerkin method as in \cite[Theorem 3.5]{tsai2018lectures}. Fix $\{(\phi_j,\psi_j)\}^{\infty}_{j=1} \subset C^{\infty}(\mathbb{T}^5;\mathbb{R}^5)$, a divergence-free orthonormal basis of $(H^1(\mathbb{T}^5;\mathbb{R}^5))^2$. Consider the finite rank approximation of $(\omega_0,\rho_0)$ in this basis:
    \begin{align*}
        \omega^0_m=\sum^m_{j=1} \phi_j(x)(\phi_j,\omega^0)_{L^2},\quad \rho^0_m=\sum^m_{j=1} \psi_j(x)(\psi_j,\rho^0)_{L^2} .
    \end{align*}
    We search for a solution of \eqref{B.1} in this basis, namely a vector field of the form
    \begin{align}\label{B.2}
        \begin{cases}
            \omega_m(x,t)=\sum^m_{j=1} g_{j,1}(t)\phi_j(x),\\
            \rho_m(x,t)=\sum^m_{j=1} g_{j,2}(t)\psi_j(x).
        \end{cases}
    \end{align}
   and desire to solve a system of ODE's for the coefficient functions $(g_{j,1}(t),g_{j,2}(t))$ with initial data given by 
   \begin{align*}
       \begin{cases}
           g_{j,1}(t)=(\phi_j,\omega^0),\\
           g_{j,2}(t)=(\psi_j,\rho^0).
       \end{cases}
   \end{align*}
   We then choose the coefficients in the \eqref{B.2} so that $(w_m,\rho_m)$ obeys \eqref{B.1}. Of course this is not possible because the partial basis in not closed under the operations in the equation; instead, we solve \eqref{B.1} projected onto the subspace. Namely,
   \begin{align*}
       (\phi_i,\sum_{j=1}^m (g_{j,1}'\phi_j-g_{j,1}\Delta\phi_j+g_{j,1}\mathrm{div}(v\otimes \phi_j+\phi_j \otimes v )-g_{j,2}\mathrm{div}(b\otimes \psi_j+ \psi_j\otimes b)))_{L^2}=0,\\
       (\psi_i, \sum_{j=1}^m (g_{j,2}'\psi_j-g_{j,2}\Delta\psi_j+g_{j,1}\mathrm{div}(\phi_j\otimes b-b \otimes \phi_j)+g_{j,2}\mathrm{div}(v\otimes \psi_j- \psi_j\otimes v)) )_{L^2}=0,
   \end{align*}
which leads to the ODEs system\\
\begin{equation*}
    \begin{cases}
        g_{i,1}'+a^1_{ij}(t)g_{j,1}-a^2_{ij}(t)g_{j,2}=0,\\
    g_{i,2}'+b_{ij}(t)g_{j,1}+b^2_{ij}(t)g_{j,2}=0,\\
    g_{j,1}(t)=(\phi_j,\omega^0),\quad g_{j,2}(t)=(\psi_j,\rho^0),
    \end{cases}
\end{equation*}
   for $i=1,2,...,m,$ where\\
   \begin{equation*}
       \begin{cases}
              a^1_{ij}(t)=(\phi_i,-\Delta\phi_j+\mathrm{div}(v\otimes \phi_j+\phi_j\otimes v)) ,\\
       a^1_{ij}(t)=(\phi_i,\mathrm{div}(b\otimes \psi_j+\psi_j\otimes b)) ,\\
       b^1_{ij}(t)=(\psi_i,\mathrm{div}(\phi_j\otimes b-b\otimes \phi_j)) ,\\
       b^2_{ij}(t)=(\psi_i,-\Delta\psi_j+\mathrm{div}(v\otimes \psi_j-\psi_j\otimes v )).
   \end{cases}
   \end{equation*}
   Because we have $(v,b) \in (L^1([t_0,\infty);C^1(\mathbb{T})))^2$, the norm of $(a^1_{ij},b^1_{ij})$ and $(a^2_{ij},b^2_{ij})$ give a unique global solution $(g_{1,1},g_{1,2}),...,(g_{m,1},g_{m,2})$ to \eqref{B.1}. Having constructed the coefficients, we can build the function $(w_m,\rho_m)$ as defined in \eqref{B.2}. By construction, $(w_m,b_m)$ obeys \eqref{B.1} projected onto any of the $(\phi_1,\psi_1),...,(\phi_m,\psi_m)$. In particular we may test against $(w_m,\rho_m)$ to obtain
   \begin{align*}
\frac{1}{2}\int_{\mathbb{T}^5} |\omega_m(t)|^2dx + \int^t_{t_0}\int_{\mathbb{T}^5}|\nabla \omega_m|^2dxdt
+\int_{t_0}^t\int_{\mathbb{T}^5} \mathrm{div}\bigl(v\otimes \omega+\omega\otimes v-b\otimes \rho-\rho\otimes b\bigr)dx = \frac{1}{2}\int_{\mathbb{T}^5}|\omega_m^0|^2dx, \\
\frac{1}{2}\int_{\mathbb{T}^5} |\rho_m(t)|^2dx + \int^t_{t_0}\int_{\mathbb{T}^5}|\nabla \rho_m|^2dxdt
+\int_{t_0}^t\int_{\mathbb{T}^5} \mathrm{div}\bigl(v\otimes \rho-\rho\otimes v+\omega\otimes b-b\otimes \omega\bigr)dx = \frac{1}{2}\int_{\mathbb{T}^5}|\rho_m^0|^2dx.
\end{align*}
   which use that $v,b,\omega_m$ and$\rho_m$ are divergence-free. We estimate the last term using the Peter-Paul inequality, we then conclude that
   \begin{align*}
   &\quad\int_{\mathbb{T}^5} |\omega_m(t)|^2+|\rho_m(t)|^2dx+\int^t_{t_0}\int_{\mathbb{T}^5}|\nabla \omega_m|^2+|\nabla\rho_m|^2 dxdt\\& \lesssim \int_{\mathbb{T}^5}|\omega^0_m|^2 +|\rho_m^0|^2 dx+\int^t_{t_0} (\|\omega_m\|^2_{L^2}+\|\rho_m\|_{L^2}^2)(\|v\|^2_{L^{\infty}}+\|b\|_{L^{\infty}}^2)dx.
   \end{align*}
   
   By Gr$\ddot{\mathrm{o}}$nwall's inequality, we get that $(\omega_m,\rho_m)$ is a priori bounded in the energy space, uniformly in $m$. Hence, there exists a limiting function $(\omega_{\infty},\rho_{\infty})(x,t)$ such that $(\omega_m,\rho_m)$ converges weak-${*}$ to $(w_{\infty},\rho_{\infty})(x,t)$ in $L^{\infty}_tL^2_x$ and converges weakly to $(\omega_{\infty},\rho_{\infty})(x,t)$ in $L^2_tL^2_x$. Moreover $(\omega_{\infty},\rho_{\infty})(x,t)$ satisfies the same priori bound as the $(\omega_m,\rho_m)$. To show that $(\omega_{\infty},\rho_{\infty})$ is a weak solution of \eqref{B.1}, it is suffice to test aganist function of the form $\bar{\phi}:=\theta_1(t)\phi_j(x)$ and $\bar{\psi}:=\theta_2(t)\psi(t)$. Hence, for all $m\ge j$ we have
   \begin{align*}
       \int^T_{t_0}\int_{\mathbb{T}^d}-\omega_m\cdot\partial_t\bar{\phi}-\omega_m\cdot\Delta\bar{\phi}-\frac{1}{2}(v\otimes \omega+\omega\otimes v-b\otimes \rho-\rho\otimes b):\nabla\bar{\phi}dx=0,\\
       \int^T_{t_0}\int_{\mathbb{T}^d}-\rho_m\cdot\partial_t\bar{\psi}-\rho_m\cdot\Delta\bar{\psi}-\frac{1}{2}(v\otimes \rho-\rho\otimes v+\omega\otimes b-b\otimes \omega):\nabla\bar{\psi}dx=0.
   \end{align*}
   Consider the weak convergence of $(\omega_m,b_m)$, we can pass to the limit and conclude $(\omega_{\infty},\rho_{\infty})(x,t)$ Satisyfies \eqref{B.1}. From the energy bounds on $(\omega_m,\rho_m)(x,t)=(\omega_{\infty},\rho_{\infty})(x,t)$, we obain that $(\partial_t\omega_{\infty},\partial_t\rho_{\infty})\in (L^2_tH^{-1}_x)^2$ exists in the weak sense, so we have $(\omega_{\infty},\rho_{\infty})\in (C([t_0,T],L^2_x))^2$ and convergence to the initial data $(\omega^0,\rho^0)$.

   Finally, we upgrade this weak solution to a strong one. One approach is to consider $(w,\rho)$ as a solution of the following system
   \begin{equation*}
   \begin{cases}
              \partial_t\omega-\Delta \omega+\nabla p=\mathrm{div}\,f,\\
       \partial_t\rho-\Delta \rho=\mathrm{div}\,g,\\
       \mathrm{div}\,\omega=\mathrm{div}\,\rho=0,\\
       \omega(t_0,x)=\omega^0(x),\quad \rho(t_0,x)=\rho^0(x),
   \end{cases}
   \end{equation*}
   where $f=\mathrm{div}(v\otimes \omega+\omega\otimes v-b\otimes \rho-\rho\otimes b) \in L^{\infty}_tL^2_x$ and $g=\mathrm{div}(v\otimes \rho-\rho\otimes v+\omega\otimes b-b\otimes \omega) \in L^{\infty}_tL^2_x$. The weak solution can be expressed as follows
   \begin{align*}
       &w(t)=e^{t\Delta}w^0+\int_0^t e^{(t-t')\Delta}\mathbb{P}\mathrm{div}\,f(s)ds,
       \\&\rho(t)=e^{t\Delta}\rho^0+\int_0^t e^{(t-t')\Delta}\mathbb{P}\mathrm{div}\,g(s)ds,
   \end{align*}
   From this formula, we can iteratively bootstrap regularity using the smoothing effect of the heat kernel.
\end{proof}

\section{Gr\"{o}nwall inequality and heat esitmate}
We recall the fractional Gr\"{o}nwall inequality from \cite[Lemma B.3]{Coiculescu2025}.
\begin{lemm}\label{B.3}
    If $a(t)$, $f(t)$, $g_1(t)$, and $g_2(t)$ are positive continuous functions on $[t_0,\infty)$
with $a(t)$ non-decreasing and
\[
f(t) \le a(t) + \int_{t_0}^t \bigl(g_1(s) + (t-s)^{-\frac{1}{2}}g_2(s)\bigr)f(s)\,ds \quad \forall t \ge t_0,
\]
then for any $p > 2$,
\[
f(t) \lesssim_p a(t)\exp\biggl(O_p\Bigl(\int_{t_0}^t g_1(s)\,ds + \|s^{\frac{1}{2}}g_2\|_{L^\infty([t_0,t])}^{p-2}\int_{t_0}^t g_2(s)^2\,ds\Bigr)\biggr)
\quad \forall t \ge t_0.
\]
\end{lemm}
Now recall the heat estimate. 
\begin{lemm}
    For all $m \ge 0$, $s_1, s_2 \in \mathbb{R}$ satisfying $m + s_1 - s_2 \ge 0$, and $f \in \mathcal{C}^{s_2}$,
we have
\begin{align}\label{2.7}
    \|e^{t\Delta} \nabla^m f\|_{\mathcal{C}^{s_1}} \lesssim_{m,s_1,s_2} t^{-\frac{m+s_1-s_2}{2}} \|f\|_{\mathcal{C}^{s_2}}
\end{align}
where $\mathcal{C}^s$ is the homogeneous H\"{o}der-Zygmund space with norm
\begin{align*}
    \| u\|_{\mathcal{C}^s} := \sup_{N \in 2^\mathbb{N}} N^s \| P_N u \|_{L^\infty} < \infty.
\end{align*}
Moreover, the above estimate holds as well if $f$ on the left-hand side is replaced by $\mathbb{P}f$. We remark that for $s=m+\alpha$ with $m\in\mathbb{N}$ and $\alpha\in(0,1)$, $\mathcal{C}^s$ coincides with the standard homogeneous H\"{o}der space $C^s$.
\end{lemm}

\smallskip
\noindent\textbf{Acknowledgments} This work was partially supported by the National Natural Science Foundation of China (No.12171493). 



\phantomsection
\addcontentsline{toc}{section}{\refname}
\bibliographystyle{abbrv} 
\bibliography{Reference}

\end{document}